\documentclass[a4paper,11pt,leqno]{article}

\usepackage{amssymb,amsmath}
\usepackage[abbrev]{amsrefs}
\usepackage{hyperref}
\usepackage{xy}
\xyoption{all}
\usepackage[text={15cm,24.5cm}]{geometry}

\numberwithin{equation}{subsection}
\newtheorem{defn}{Definition}[section]
\newtheorem{corollary}[defn]{Corollary}
\newtheorem{rem}[defn]{Remark}
\newtheorem{exm}[defn]{Example}
\newtheorem{lemma}[defn]{Lemma}
\newtheorem{theorem}[defn]{Theorem}
\newtheorem{xproof}{{\it Proof. }}

\newenvironment{definition}{\begin{defn}\em}{\end{defn}}
\newenvironment{remark}{\begin{rem}\em}{\end{rem}}
\newenvironment{example}{\begin{exm}\em}{\end{exm}}
\newenvironment{proof}{\begin{xproof}\em}{\end{xproof}}

\def\qed{\hspace{0.3cm}{\rule{1ex}{2ex}}}
\newcommand\V{\bigvee}
\newcommand\ie{i.e.}
\newcommand\eg{e.g.}
\newcommand\st{\mid}
\newcommand\cf{\textrm{cf.}}
\newcommand\opens{\operatorname{\mathcal{O}}}
\newcommand\topology{\operatorname{\Omega}}
\newcommand\spp{\operatorname{\varsigma}}
\newcommand\downsegment{\operatorname{\downarrow}}
\newcommand\Frm{\textit{Frm}}
\newcommand\Loc{\textit{Loc}}
\newcommand\sets{\textit{Sets}}
\newcommand\ident{\mathrm{id}}
\newcommand\ipi{\operatorname{\mathcal I}}
\newcommand\germ{\operatorname{germ}}
\newcommand\lcc{\operatorname{{\mathcal L}^{\vee}}}
\newcommand\Germs{\operatorname{Germs}}
\newcommand\intg[1]{\widetilde{#1}}
\newcommand\lres[2]{{_{#1}{\vert}{#2}}}
\newcommand\rres[2]{{{#1}{\vert}_{#2}}}
\newcommand\sections{\mathit{\Gamma}}
\newcommand\biGrpd{\mathbb{GRPD}}
\newcommand\biQOl{\mathbb{QOL}}
\newcommand\cov{\widehat}
\newcommand\act{\mathfrak a}
\newcommand\bct{\mathfrak b}
\newcommand\opp[1]{{#1}^{\mathrm{op}}}
\newcommand\rs{\operatorname{\mathsf{R}}}
\newcommand\LH{\textit{LH}}
\newcommand\Sh{\textit{Sh}}

\begin{document}

\title{Actions of \'etale-covered groupoids}

\author{Juan Pablo Quijano and Pedro Resende\thanks{Work funded by FCT/Portugal through the LisMath program and project PEst-OE/EEI/LA0009/2013.}}

\maketitle

\begin{abstract}
By restricting to a class of localic open groupoids $G$ which, similarly to Lie group\-oids, possess appropriate covers $\cov G\to G$ by \'etale group\-oids, we extend results about groupoid actions and quantales that were previously proved for \'etale groupoids but do not seem to work for arbitrary open groupoids. In particular we obtain a characterization of the category of $G$-actions as a category of quantale modules on $\opens(\cov G)$ that satisfy a condition related to the quantale $\opens(G)$. This  leads to a simple description of $G$-sheaves and the classifying topos of $G$ in terms of Hilbert $\opens(\cov G)$-modules. The bicategory whose 1-cells are the groupoid bi-actions is bi-equivalent to a corresponding bicategory of quantales and bimodules.
\\
\vspace*{-2mm}~
\vspace*{-2mm}~
\end{abstract}

\tableofcontents


\section{Introduction}

It is well known that inverse semigroups are closely related to \'etale groupoids (see, \eg, \cites{Paterson,RenaultLNMath,Kumjian84}), and in \cite{Re07} 
quantales, more specifically inverse quantal frames, were put forward as mediating objects between the semigroups and the groupoids. In part the aim was to bring the correspondence to bear on localic groupoids rather than just topological groupoids, in particular making it constructive in a topos-theoretic sense, but also, independently, to provide an alternative algebraic language with which to describe \'etale groupoids, with the quantales being regarded as ring-like objects. This developed naturally into a program where various constructions for \'etale groupoids, such as actions and sheaves, are translated to quantale modules~\cites{RR,GSQS}. One difficulty is that, while the correspondence between the ensuing categories of actions of groupoids and modules on quantales is functorially well behaved, the actual correspondence between \'etale groupoids and their quantales is not, at least not in the sense of a direct correspondence between groupoid functors and quantale homomorphisms~\cite{Re07}, unless one severely restricts the class of functors~\cite{LL}. This problem is circumvented by considering bicategories and functoriality in the form of a bi-equivalence where the morphisms (1-cells) are groupoid bi-actions and quantale bimodules~\cite{Re15}, which furthermore enables one to `explain' why other notions of morphism between groupoids relate well to quantale homomorphisms --- just as they are known to relate well to $*$-homomorphisms of C*-algebras~\cite{Bun08} and to homomorphisms of inverse semigroups~\cite{BEM12}.

Another direction of research pertains to more general groupoids (or categories~\cite{KL}, toposes~\cite{HeymansGrQu}, etc.) that allow such algebraic treatments. Here the direct correspondence to inverse semigroups breaks down, but the relation to quantales does not, and it is the quantales that can be regarded as the natural algebraic language for handling general open groupoids~\cites{PR12,QRnonuspp}. However, there are additional difficulties that did not exist in the \'etale case. In particular the direct correspondences between categories of groupoid actions and categories of quantale modules are less well behaved, with only a functor from groupoid actions to quantale modules being readily available (rather than a full-fledged equivalence of categories), as was already observed in~\cite{GSQS}.

The aim of this paper is to recover the good behavior of actions, sheaves, and functoriality of \'etale groupoids by restricting to a smaller class of open groupoids that was introduced in~\cite{PR12}, which nevertheless is sufficiently large, at least for applications in analysis and differential geometry, because it includes many locally compact groupoids and, in particular, Lie groupoids. Such groupoids $G$ are equipped with good pseudogroups of local bisections, which in turn lead to a certain notion of cover $J:\cov G\to G$ by an \'etale groupoid, so here we call them \emph{\'etale-covered groupoids}. Roughly, it is the existence of \'etale covers that provides the good behavior of actions and sheaves for such groupoids, as we shall see. Dually to such groupoids there is the notion of \emph{inverse-embedded quantal frame}, consisting of an inverse quantal frame $\cov \opens$ with a suitable (usually non-multiplicative) embedding $j:\opens\to \cov\opens$ of a non-unital quantale. These had already appeared in~\cite{PR12}, but in the present paper we shall need to study them in more detail, in particular taking advantage of some notions and results from~\cite{QRnonuspp}.

The main results of this paper are those of section~\ref{section:actionsetalecoveredgprds}. This addresses actions, namely relating the actions of an \'etale-covered  groupoid $G$ to the modules of the quantale $\opens(\cov G)$ that in addition behave well with respect to the embedding $j:\opens(G)\to\opens(\cov G)$. This leads to an equivalence between the category of $G$-actions and the category of such $\opens(\cov G)$-modules, and extends the equivalence of categories that exists if $G$ is \'etale. In addition, this section contains two applications of these results. One is a description of $G$-sheaves in terms of $\opens(\cov G)$-modules that extends that of the \'etale case, whereby a $G$-sheaf $X$ is shown to correspond to an $\opens(\cov G)$-sheaf whose inner product $\langle-,-\rangle:X\times X\to\opens(\cov G)$ is valued in the image $j(\opens(G))$. This is a remarkably simple axiom that resembles a continuity condition. As a consequence, we are provided with a representation of the classifying topos of any \'etale-covered groupoid in terms of sheaves on its inverse-embedded quantal frame. The second application is an extension of the functoriality results of~\cite{Re15}, ultimately yielding a biequivalence between the bicategory of \'etale-covered groupoids and the bicategory of inverse-embedded quantal frames. Extensions of later functoriality results, namely those of~\cite{Funct2} regarding Hilsum--Skandalis maps and Morita equivalence, face additional difficulties and will not be addressed in this paper.
 

\section{Preliminaries}

The purpose of this section is to recall some concepts and to fix terminology and notation, mostly following~\cites{GSQS,QRnonuspp,Funct2,Re07}. For general references on sup-lattices, locales, quantales, or groupoids see~\cites{Rosenthal1,gamap2006,stonespaces,JT,EGHK18}.


\subsection{Groupoid quantales}

This section provides some background on groupoid quantales and their sheaves, based on \cite{Re07,GSQS,QRnonuspp}, which will be needed in this paper.


\paragraph{Open groupoids.}

The structure maps of a groupoid $G$ will be denoted as follows:
\[
\xymatrix{
G=\quad G_2\ar[r]^-m&G_1\ar@(ru,lu)[]_i\ar@<1.2ex>[rr]^r\ar@<-1.2ex>[rr]_d&&G_0.\ar[ll]|u
}
\]
Here $G_2$ is the pullback of the \emph{domain map} $d$ and the \emph{range map} $r$ in the category of locales $\Loc$. The groupoid is \emph{open} if $d$ is an open map, in which case $r$ and the \emph{multiplication map} $m$ are open, too. The \emph{inversion map} $i$ is always an isomorphism in $\Loc$. We shall denote the quantale of $G$ by $\opens(G)$, as in~\cite{Re07}. But we shall not make any distinction between a locale $X$ regarded as an object of $\Loc$ or as an object of $\Frm = \opp\Loc$. We say $G$ is \emph{\'etale} if $d$ is a local homeomorphism, in which case all the structure maps are local homeomorphisms.


\paragraph{Based quantales.}

Let $B$ be a locale. A \emph{$B$-$B$-bimodule} $M$ is a sup-lattice equipped with two unital (resp.\ left and right) $B$-module structures $B\times M\to M$ and $M\times B\to M$,
\[
(a,m)\mapsto \lres{a}{m}\quad \text{and}\quad (m,a)\mapsto \rres{m}{a},
\]
satisfying the following additional condition for all $a, b \in B$
and $m \in M$:
\begin{align}
  \rres{(\lres{a}{m})}{b} & = \lres{a}{(\rres{m}{b})}.
\end{align}
The notation $\lres{a}m$ for the left action is meant to convey the idea that $a$ restricts $m$ on the left, and analogously for the right action.

By a \emph{quantale based on $B$}, or a \emph{$B$-$B$-quantale}, will be meant a $B$-$B$-bimodule $Q$ equipped with a quantale multiplication $(x, y)\mapsto xy$ that satisfies the following additional conditions
for all $a \in B$ and $x, y \in Q$:
\begin{align}
 (\lres{a}{x})y & =  \lres{a}{(xy)}, \\
 (\rres{x}{a})y & =   x(\lres{a}{y}),  \\
 \rres{(xy)}{a} & =   x(\rres{y}{a}).
\end{align}


\paragraph{Involutive based quantales.}

A $B$-$B$-quantale is \emph{involutive} if it is an involutive semigroup; the involution is denoted by $a\mapsto a^{*}$ and it is required to satisfy, besides the standard conditions $x^{**}=x$ and $(xy)^{*}=y^{*}x^{*}$, the following two conditions:
\begin{align}
 (\bigvee_i x_i)^{*} & =\bigvee_i x^{*}_i, \\
 (\lres{a}{\rres{x}{b}})^{*} & = \lres{b}{\rres{x^{*}}{a}}.
\end{align}


\paragraph{Equivariant supports.}

An involutive $B$-$B$-quantale $Q$ is \emph{supported} if it is equipped with a sup-lattice homomorphism $\spp: Q\rightarrow B$ satisfying the following conditions for all $x,y\in Q$:
\begin{align}
\spp(1_Q)&=1_B,  \label{spp1}\\
\lres {\spp(x)}{y} &\leq xx^{*}y, \label{spp2}\\
\lres {\spp(x)}{x}&= x. \label{spp3}
\end{align}
A \emph{supported $B$-$B$-quantale} $(Q,\spp)$, also referred to as a \emph{supported quantale with base locale $B$}, is an involutive $B$-$B$-quantale equipped with a specified support $\spp$. The support $\spp$ is said to be \emph{equivariant} if for all $a\in B$ and $x\in Q$
\begin{equation}
\spp(\lres{a}{x})=a\wedge \spp(x).
\end{equation}
If a support is equivariant then it is the only possible support, and in this case the
$B$-$B$-quantale $(Q,\spp)$ is said to be \emph{equivariantly supported}.

Any equivariant support is necessarily \emph{stable}, by which it is meant that the following equivalent conditions hold:
\begin{itemize}
\item $\spp(xy)\le \spp(x)$ for all $x,y\in Q$;
\item $\spp(x1_Q)\le\spp(x)$ for all $x\in Q$;
\item $\spp(xy) = \spp(\rres{x}{\spp(y)})$ for all $x,y\in Q$.
\end{itemize}
An involutive $B$-$B$-quantale is \emph{stably supported} if it is equipped with a stable support.

If $Q$ is equivariantly supported with base locale $B$ then, writing $\rs(Q)$ for the set of right-sided elements of $Q$, where an element $a\in Q$ is \emph{right-sided} if
\begin{equation}\label{def:rs}
a1_Q\leq a,
\end{equation}
the map $B\to \rs(Q)$ defined by $x\mapsto \lres{x}{1_Q}$ is an order isomorphism whose inverse is the map $\rs(Q) \to B$ defined by $x\mapsto \spp(x)$. So $\rs(Q)\cong B$, and thus $\rs(Q)$ is a locale.


\paragraph{Based quantal frames.}

By a \emph{$B$-$B$-quantal frame} is meant a $B$-$B$-quantale $Q$ such that for all $q,m,m_i\in Q$ and $a\in B$ the following properties hold:
\begin{align}
q \wedge \bigvee_i m_i &= \bigvee_i q\wedge m_i,\\
(\lres{a}{q})\wedge m &= \lres{a}{(q\wedge m)},\\
m\wedge (\rres{q}{a}) &= \rres{(q\wedge m)}{a}.
\end{align}


\paragraph{Reflexive quantal frames.}

By a \emph{reflexive quantal frame} $(Q,\upsilon)$ is meant a $B$-$B$-quantal frame equipped with a frame homomorphism $\upsilon: Q\to B$ such that for all $a\in B$
\begin{equation}
\upsilon(\lres{a}{1_Q})=a=\upsilon(\rres{1_Q}{a}).
\end{equation} 


\paragraph{Multiplicative quantal frames.}

Let $Q$ be a $B$-$B$-quantal frame.
The quantale multiplication has the following factorization in the category of sup-lattices:
\[
\vcenter{\xymatrix{
 Q\otimes Q \ar@{->>}[d]_{\pi}  \ar[dr] \\
 Q\otimes_{B} Q  \ar[r]_-{\mu} & Q.
}
}
\]
We refer to $\mu$ as the \emph{reduced multiplication} of $Q$, and to its right adjoint $\mu^*$ as the \emph{reduced comultiplication}. Then $Q$ is said to be a \emph{multiplicative quantal frame} if the reduced comultiplication preserves joins (and therefore is a homomorphism of locales).


\paragraph{Quantales of open groupoids.}

Let $(Q,\spp,\upsilon)$ be a multiplicative equivariantly supported reflexive quantal frame. We say that $Q$ satisfies \emph{unit laws} if moreover the following condition holds for all $a\in Q$:
\begin{equation}\label{unitlaws}
\bigvee_{xy\leq a} (\lres{\upsilon(x)}{y}) = a.
\end{equation}
By a \emph{groupoid quantale} will be meant a multiplicative equivariantly supported reflexive quantal frame $Q$ that satisfies unit laws and moreover satisfies the following condition, which is referred to as the \emph{inverse law}, for all $a\in Q$:
\begin{equation}\label{inverselaws}
\lres{\upsilon(a)}{1_Q} = \bigvee_{xx^*\leq a} x.
\end{equation}
The groupoid quantales are precisely the quantales $Q\cong \opens(G)$ for an open groupoid $G$.


\paragraph{Inverse quantal frames.}

Let $Q$ be an equivariantly supported reflexive $B$-$B$-quantal frame. If moreover $Q$ is a unital quantale and it satisfies the inverse law then $Q$ is necessarily multiplicative and it satisfies the unit laws. In this case $Q$ is an inverse quantal frame, in other words a quantale $Q\cong\opens(G)$ for an \'etale groupoid $G$. Among other things, we have
\begin{equation}
\bigvee \ipi(Q)= 1_Q,
\end{equation}
where $\ipi(Q) = \{a \in Q \st a^*a\vee a^*a\leq e\}$ is the set of partial units of $Q$, and there is an order isomorphism $\iota:B\to\downsegment(e)$ that transforms the $B$-actions into multiplication:
\[
\lres{a}x = \iota(a)x\quad\quad\textrm{and}\quad\quad\rres{x}a = x\iota(a).
\]
For an arbitrary inverse quantal frame $Q$ we shall usually refer to the locale $\downsegment(e)$ as the \emph{base locale}.


\subsection{Actions and sheaves}

\paragraph{Actions of open groupoids.}

Let $G$ be an open groupoid. A \emph{left $G$-action} is a triple $(X,p,\act)$ where $X$ is a locale, $p: X\to G_0$ (called \emph{projection} or \emph{anchor map}) is a map of locales, and $\act: G_1\times_{G_0} X \to X$ is a map of locales (called \emph{the action}) that satisfies the usual axioms (see, \eg, \cite{GSQS}). One defines \emph{right $G$-actions} in a similar way. We shall denote $(X,p,\act)$ by $X$ when no confusion will arise. The category of $G$-locales and equivariant maps between them is denoted by $G$-$\Loc$.  The categories of left $G$-locales and right $G$-locales are isomorphic. 

We shall denote by $X$ the $\opens(G)$-module which is obtained from a $G$-locale $X$. Let us briefly recall the construction of this module. Taking into account that in $\Frm$ the locale $G_1\times_{G_0} X$ is a quotient
$\opens(G_1)\otimes_{\opens(G_0)} X$ of the tensor product
$\opens(G_1)\otimes X$, the module action is the sup-lattice homomorphism which is obtained as the following composition:
\[
\vcenter{\xymatrix{
 \opens(G_1)\otimes X   \ar@{->>}[r]  & \opens(G_1)\otimes_{\opens(G_0)} X \ar[r]^-{\act_!} & X.
 }
}
\]
The inverse image homomorphism $\act^*$ is the right adjoint of $\act_!$, and thus it is given by
\begin{eqnarray}
\act^*(x) &=& \V \{ a\otimes y\st ay\leq x   \}. \label{eq:rightadjformula1}
\end{eqnarray}
Moreover, when $G$ is an \'etale groupoid, we have the following useful formula:
\begin{eqnarray}
\act^*(x) &=& \V_{s\in \ipi(Q)} s\otimes s^*x. \label{eq:rightadjformula2}
\end{eqnarray}
The latter shows that $\act^*$ preserves arbitrary joins.  


\paragraph{Sheaves.}

Let $G$ be an \'etale groupoid. A \emph{$G$-sheaf} is a $G$-locale whose projection is a local homeomorphism. The full subcategory of $G$-$\Loc$ whose objects are the $G$-sheaves (the classifying topos of $G$) is usually denoted by $BG$. The isomorphism between $G$-$\Loc$ and $\opens(G)$-$\Loc$ (see \cite{GSQS}*{Th.\ 3.21}) yields, by restriction, a corresponding full subcategory of \emph{$\opens(G)$-sheaves}. Concretely, letting $Q$ be an arbitrary inverse quantal frame, a \emph{$Q$-sheaf} is a left $Q$-locale $X$ whose action restricted to the locale $B:=\downsegment(e)$ defines a $B$-sheaf. The full subcategory of $Q$-$\Loc$ whose objects are the $Q$-sheaves is denoted by $Q$-$\LH$.

Let $X$ be a $Q$-sheaf. The \emph{local sections} of $X$ are the local sections of
$X$ regarded as a $B$-sheaf; that is, a local section is an element $s\in X$
satisfying $\spp_X(x)s=x$ for all $x\le s$.
The set of local sections of $X$ is denoted by $\sections_X$.

Let $Q$ be an inverse quantal frame, and let $X$ and $Y$ be $Q$-sheaves. A \emph{sheaf homomorphism} $f:X\to Y$ is a left $Q$-module homomorphism satisfying the following two conditions:
\begin{enumerate}
\item $\spp_Y(f(x))=\spp_X(x)$ for all $x\in X$;
\item $f(\sections_X)\subset\sections_Y$.
\end{enumerate}
The sheaf homomorphisms coincide with the direct image homomorphisms of the morphisms in $Q$-$\LH$. Therefore the category of $Q$-sheaves and sheaf homomorphisms between them, which we shall denote by $Q$-$\Sh$, is isomorphic to $Q$-$\LH$.


\paragraph{Hilbert modules.}

Let $Q$ be an involutive quantale. By a \emph{pre-Hilbert $Q$-module} will be meant a left $Q$-module $X$ equipped with a binary operation $\langle -,-\rangle: X\times X\to Q$,
called the \emph{inner product}, which for all $x, x_i, y \in X$ and $a \in Q$ satisfies the
following axioms:
\begin{align}
\langle ax,y\rangle &= a\langle x,y\rangle\\
\bigl\langle\V_i x_i,y \bigr\rangle &= \V_i \langle x_i,y \rangle\\
\langle x,y\rangle &=\langle y,x \rangle^*.
\end{align}

By a \emph{Hilbert $Q$-module} will be meant a pre-Hilbert $Q$-module whose inner
product is non-degenerate:
\begin{align}
\langle x,-\rangle = \langle y,-\rangle \Rightarrow x=y.
\end{align}
In particular, inner products are sesquilinear forms.

Any set $\sections \subset X$ such that $x=\V_{t\in \sections} \langle x,t \rangle t$ for all $x \in X$ is called a \emph{Hilbert basis}. If $X$ has a Hilbert basis we say that the Hilbert module is \emph{complete}.
By a \emph{Hilbert section} of $X$ is meant an element $s \in X$ such that $\langle x, s\rangle s \leq x$ for all $x \in X$. In particular, any element of a Hilbert basis is a Hilbert section.

We recall the Hilbert module characterization of quantale sheaves for inverse quantal frames (see~\cite{GSQS}*{Th.\ 4.47, Th.\ 4.55}).

\begin{theorem}\label{theo: Hilbertisosheaves}
For any inverse quantal frame $Q$, complete Hilbert $Q$-modules and $Q$-sheaves amount to the same thing, and the local sections of a $Q$-sheaf coincide with the Hilbert sections.
\end{theorem}

The following theorem gives a useful formula for computing the inner products of quantale sheaves for inverse quantal frames (see~\cite{Funct2}*{Th.\ 3.6}).

\begin{theorem}\label{theorem:innerproductformula}
Let $Q$ be an inverse quantal frame and $X$ be a $Q$-sheaf. Then 
\begin{equation}\label{preliminaries, eq: innerproduct}
\langle x,y\rangle=\V_{u\in Q_{\mathcal{I}}} u \spp_X(u^*x \wedge y),
\end{equation}
for all $x,y\in X$.
\end{theorem}


\paragraph{Supported modules.}

Let $(Q,\spp_Q)$ be a unital supported quantale, and denote by $B$ the locale $\downsegment(e)$. By a \emph{supported $Q$-module} is meant a pre-Hilbert $Q$-module $X$ equipped with a monotone map $\spp_X: X\to B$, called the \emph{support} of $X$, such that the following properties hold for all $x\in X$:
\begin{align}
\spp_X(x) &\leq \langle x, x \rangle \\
x &\leq \spp_X(x)x.
\end{align}
Note that any supported quantale $Q$ defines a supported module over
itself, with $ \langle a, b\rangle = ab^*$. The support is called \emph{stable} if in addition one of the following equivalent conditions holds for all $b \in B$ and $x \in X$:
\begin{align}
\spp_X(bx) &=b\wedge \spp_X(x)\\
\spp_X(bx) &= \spp_Q(b\spp_X(x))\\
\spp_X(bx) &\leq \spp_Q(b).
\end{align}
Moreover, if $(Q,\spp_Q,e)$ is a stably supported quantale then any supported $Q$-module $X$ is necessarily stably supported and the following properties hold for all $x,y\in X$ and $a\in Q$:
\begin{align}
\spp_Q(\langle x, y\rangle) &\leq \spp_X(x) = \spp_Q(\langle x, x\rangle) = \spp_Q(\langle x, 1_X\rangle)\\
\spp_X(x)a  &= \langle x, 1_X \rangle \wedge a\\
\spp_X(x) &= \langle x, 1_X\rangle \wedge e = \langle x, x\rangle \wedge e.
\end{align}
Therefore, for any stably supported quantale $(Q,\spp_Q,e)$, any complete Hilbert $Q$-module is a (necessarily stably) supported $Q$-module.


\section{\'Etale covers}\label{funct3, sec: inverseembedded}

Here we introduce the main definitions of this paper, namely inverse-embedded quantales and their \'etale-covered groupoids.

\subsection{Local bisections}

We begin by recalling (and adapting to the setting of $B$-$B$-quantales) the notion of local bisection of \cite{PR12} for open quantal frames, which generalizes the corresponding notion for groupoids.

Let $\opens$ be a groupoid quantale with base locale $B$. By a \emph{local bisection} of $\opens$ will be meant a pair $\sigma=(U,s)$ where $U\in B$ and
\begin{equation}\label{def:bisection}
s: \widetilde U\to \opens\quad\quad \text{(with $\widetilde U:= \downsegment U\cap B$)}
\end{equation}
is a map of locales such that
\begin{enumerate}
\item $d\circ s=k_U$, where $k_U:\widetilde U\to B$ is the inclusion of the open sublocale $\widetilde U$ into $B$ ($s$ is a local section of $d$),
\item and $r\circ s$ is an open regular monomorphism of locales.
\end{enumerate}

The notion of local bisection for an open quantal frame $\opens$, along with a corresponding action of the local bisections on $\opens$, is used in~\cite{PR12} in order to define a weak form of multiplicativity which ensures that the set of local bisections has the structure of a pseudogroup $\sections(\opens)$. Then sufficient (but not necessary) conditions that ensure multiplicativity are studied. These conditions concern the extent to which $\opens$ can be embedded into the inverse quantal frame $\lcc(\sections(\opens))$. Such conditions, applied to the quantale $\opens(G)$ of an open groupoid, imply the existence of a surjective functor of groupoids $J : \cov G\to G$ that provides a notion of canonical ``\'etale cover" of $G$. 

These results provided the inspiration for the work in the present section. In particular Definition~\ref{def:ieqf} will be seen to include conditions that are not found in \cite{PR12} but are necessary in order to obtain groupoids from such quantale embeddings.


\subsection{Inverse-embedded quantal frames}

\begin{definition}\label{funct3, def: QQquantale}
Let $Q$ be an inverse quantal frame. By an involutive \emph{$Q$-$Q$-quantale} $\opens$ is meant a $Q$-$Q$-bimodule, whose left and right actions are denoted by $(a,x)\mapsto a\cdot x$ and $(x,a)\mapsto x\cdot a$, respectively, equipped with a quantale multiplication $(x, y)\mapsto xy$ that satisfies the following additional conditions
for all $a \in Q$ and $x, y \in \opens$,
\begin{align}
 (a\cdot x)y & =  a\cdot (xy),\label{QQquantale1}\\
 (x\cdot a)y & =   x(a\cdot y),\label{QQquantale2}\\
 (xy)\cdot a & =   x(y\cdot a),\label{QQquantale3}
\end{align}
and moreover is endowed with an involution $x\mapsto x^*$, by which is meant a map such that for all $x\in \opens$ we have 
\[
x^{**}=x\quad \text{and}\quad (xy)^{*}=y^{*}x^{*}
\]
plus the following two conditions:
\begin{align}
 (\V_i x_i)^{*} & =\V_i x^{*}_i,& \text{($x_i\in \opens$)} \\
 (a\cdot x \cdot b)^{*} & = b^*\cdot x^* \cdot a^*.& \text{($a,b\in Q$ and $x\in \opens$)}
\end{align}
 \end{definition}
 
 \begin{remark}\label{remark:ieqBBquantale}
Let $Q$ be an inverse quantal frame with base locale $B$. Of course, any involutive $Q$-$Q$-quantale is also an involutive $B$-$B$-quantale.
  \end{remark}

\begin{definition}\label{def:ieqf}
By an \emph{inverse-embedded quantal frame} is meant an involutive quantal frame $\opens$ (with reduced multiplication $\mu$) equipped with 
\begin{enumerate}
\item an inverse quantal frame $\cov\opens$ (with base locale $B$ and reduced multiplication $\cov\mu$), 
\item a structure of involutive $\cov\opens$-$\cov\opens$-quantale, and
\item a frame monomorphism $j:\opens\to \cov\opens$ which is a homomorphism of $\cov\opens$-$\cov\opens$-bimodules, such that
\begin{enumerate}
\item\label{def:ieq1} $j(a^*)=j(a)^*$ for all $a\in \opens$, that is, $j$ preserves the involution, 
\item\label{def:ieq2} $\xymatrix{
\opens\otimes_{B}\opens \ar[r]^-{j\otimes \ident} &\cov \opens\otimes_{B} \opens
}$   is also a monomorphism of frames,
\item\label{def:ieq3} $\cov \mu^*\circ j=(j\otimes j)\circ \mu^*$, that is, $j$ preserves the reduced comultiplications, 
\item\label{def:ieq4} $(j(a)\wedge e)1_{\cov\opens}\le \V\limits_{xx^*\leq a} j(x)$ for all $a\in \opens$, and
\item\label{def:ieq0} $\rs(\cov\opens)\subset j(\opens)$, that is, $j$ is ``right-sided surjective'' [\cf\ \eqref{def:rs}].
\end{enumerate}
\end{enumerate}
\end{definition}

\begin{example}
\begin{enumerate}
\item Let $\opens$ be a weakly multiplicative quantale in the sense of \cite{PR12}. Then $\opens$ is an inverse-embedded quantal frame with inverse quantal frame $\cov\opens=\lcc(\sections(\opens))$ and a frame monomorphism $j$ given by 
\[
j(q)=\V \{ \sigma\in \sections(\opens(G))  \st\ s^*(q)=U \},
\]
for all $q\in \opens(G)$ (see \cite{PR12}*{Lemma 5.9, Lemma 5.13}). 
\item Let $X$ be a locally compact topological space. Then the topology $\topology(\intg X)$ of the pair groupoid of $X$ is an inverse-embedded quantal frame with inverse quantal frame $\topology(\Germs(\intg X))$ (see \cite{MaRe10}*{Th.\ 2.8}) and a frame monomorphism $j=k^{-1}$ where $k$ is given by $k((x,\germ_x s))=s(x)$ (see \cite{PR12}*{Th.\ 5.31}).
\end{enumerate}
\end{example}

\begin{lemma}\label{lemma:jpreserveslaxmu}
Let $\opens$ be an inverse-embedded quantal frame. Then, for all $x,y\in \opens$, we have 
\[
\cov\mu(j(x)\otimes j(y))\le j(\mu(x\otimes y)).
\]
\end{lemma}

\begin{proof}
The following sequence of (in)equalities for all $x,y\in \opens$ will give us the desired result:
\begin{align*}
\cov\mu(j(x)\otimes j(y)) &= \cov\mu((j\otimes j)(x\otimes y))\\
&\le \cov\mu((j\otimes j)(\mu^*(\mu(x\otimes y))))& \text{(because $\mu^*\circ\mu\geq \ident$)}\\
&= \cov\mu(\cov\mu^*(j(\mu(x\otimes y))))& \text{[due to Definition~\ref{def:ieqf}\eqref{def:ieq3}]}\\
&\le  j(\mu(x\otimes y))& \text{(because $\cov\mu\circ\cov\mu^*\leq \ident$)}. \qed
\end{align*} 
\end{proof}

For the sake of simplicity, we shall write $j(x)j(y)\leq j(xy)$ rather than $\cov\mu(j(x)\otimes j(y))\le j(\mu(x\otimes y))$ whenever no confusion may arise.

\begin{lemma}\label{lemma:propertiesj}
Let $\opens$ be an inverse-embedded quantal frame. Then, for all $x,y\in \opens$, we have:
\begin{align}
j(x1_{\opens})1_{\cov\opens} &= j(x1_{\opens})&  \text{[\ie, $j(x1_{\opens})\in \rs(\cov\opens)$]} \label{eq:jprop1}\\
j(x)1_{\cov\opens} &\le j(x1_{\opens})\label{eq:jprop2}\\
j(x1_{\opens})\cdot 1_{\opens}&= x1_{\opens}\label{eq:jprop3}\\
j(x)\cdot y &\le xy \label{eq:jprop4}\\
\cov\mu^*(j(q))&= \V_{\substack{u,v\in \ipi( \cov \opens) \\ v\leq j(q)}} u\otimes u^*v \label{funct3, eq: multiplicativityopens}
\end{align} 
\end{lemma}

\begin{proof}
\eqref{eq:jprop1}: we have
$
j(x1_{\opens})1_{\cov\opens}=j(x1_{\opens})j(1_{\opens}) \leq j(x1_{\opens}1_{\opens}) \leq j(x1_{\opens})
$,
so $j(x1_{\opens})1_{\cov\opens}=j(x1_{\opens})$
because the right-sided elements of $\rs(\cov\opens)$ are strict.

\eqref{eq:jprop2}: $j(x)1_{\cov\opens} = j(x)j(1_{\opens}) \leq j(x1_{\opens})$.

\eqref{eq:jprop3}:  we have
$
j(j(x1_{\opens})\cdot 1_{\opens}) = j(x1_{\opens})j(1_{\opens}) = j(x1_{\opens})1_{\cov\opens} = j(x1_{\opens})
$,
so $j(x1_{\opens})\cdot 1_{\opens} = x1_{\opens}$ because $j$ is monic.

\eqref{eq:jprop4}: $j(j(x)\cdot y) = j(x)j(y) \leq j(xy)$, so $j(x)\cdot y \leq xy$ because $j$ is monic.

\eqref{funct3, eq: multiplicativityopens}: We have
\[
\cov\mu^*(j(q))=\cov\mu^*(\V_{\substack{v\in \ipi( \cov \opens) \\ v\leq j(q)}}v)= \V_{\substack{v\in \ipi( \cov \opens) \\ v\leq j(q)}}\cov\mu^*(v)= \V_{\substack{v\in \ipi( \cov \opens) \\ v\leq j(q)}}\V_{u\in\ipi(\cov\opens)} u\otimes u^*v,
\]
where the last equality follows from \cite{GSQS}*{Lemma 3.15}. \qed
\end{proof}

\begin{remark}\label{remark:j(opens)congopens}
We remark that in fact the conditions $j(x)j(y)\leq j(xy)$ and $j(x)\cdot y\leq xy$ are equivalent because $j(x)j(y)=j(j(x)\cdot y)$ holds for all $x,y\in \opens$. Moreover, notice that $j(\opens)$ can be made an involutive quantale isomorphic to $\opens$ because $j:\opens\to j(\opens)$ is a frame isomorphism and thus we can define the following multiplication $\bullet$ in $j(\opens)$:
\[
j(a)\bullet j(b) := j(ab).
\]
However, $\bullet$ does not coincide with the multiplication in $\cov\opens$, so $(j(\opens),\bullet)$ is in general not a subquantale of $\cov\opens$.
\end{remark}

By Lemma~\ref{lemma:propertiesj}, for all $x\in \opens$ we have  $j(x1_{\opens})1_{\cov\opens} = j(x1_{\opens})$, so $j$ restricts to a frame monomorphism $j':\rs(\opens)\to \rs(\cov\opens)$.

\begin{lemma}\label{lemma:j'isiso}
Let $\opens$ be an inverse-embedded quantal frame. Then $\opens$ satisfies the following conditions:
\begin{enumerate}
\item\label{eq:j'epi1} $j'$ is surjective;
\item\label{eq:j'epi2} $j'$ is an order isomorphism;
\item\label{eq:j'epi3} $\rs(\cov\opens)\subset j(\rs(\opens))$.
\end{enumerate}
\end{lemma}

\begin{proof} 
It suffices to prove that \eqref{eq:j'epi1} holds (clearly, conditions \eqref{eq:j'epi2} and \eqref{eq:j'epi3} are equivalent to \eqref{eq:j'epi1}). In order to see that $j'$ is surjective first notice that, by Definition~\ref{def:ieqf}\eqref{def:ieq0}, for all $y\in \rs(\cov\opens)$ there is $x\in\opens$ such that $y=j(x)$. The conclusion follows from the fact that necessarily $x=x1_{\opens}$ (\ie, $x$ is right-sided), since $j$ is monic and we have
\[
j(x) = j(x) 1_{\cov\opens} \le j(x 1_{\opens}) 1_{\cov\opens} = j(x 1_{\opens}). \qed
\]
\end{proof}

\begin{lemma}\label{lemma:supportofopens}
Let $\opens$ be an inverse-embedded quantal frame with base locale $B$. Then $\opens$ is an equivariantly supported $B$-$B$-quantal frame.
\end{lemma}

\begin{proof}
$\cov \opens$ is an inverse quantal frame with base locale $B=\downsegment (e)\cong \spp(\cov \opens)$. Hence, by Remark~\ref{remark:ieqBBquantale}, $\opens$ is an involutive $B$-$B$-quantal frame. Let us denote by $\cov\spp$ the support of $\cov \opens$, and let us verify that the sup-lattice homomorphism
\[
\spp:=\cov\spp\circ j:\opens\to B
\]
defines a support on $\opens$.
\begin{description}
\item{\eqref{spp1}:} $\spp(1_{\opens})=\cov\spp(j(1_{\opens}))=\cov\spp(1_{\cov\opens})=1_B$ because $j$ is a frame homomorphism.
\item{\eqref{spp2}:} This follows from the sequence of (in)equalities:
\begin{align*}
\lres {\spp(x)}{y} &=\lres {\cov\spp(j(x))}{y}\\
&\le  j(x)j(x)^*\cdot y& \text{[due to \eqref{spp2}]}\\
&=  j(x)j(x^*)\cdot y&\text{[due to Definition~\ref{def:ieqf}\eqref{def:ieq1}]}\\
&\le  j(xx^*)\cdot y&\text{(by Lemma~\ref{lemma:jpreserveslaxmu})}\\
&\le xx^*y& \text{[by \eqref{eq:jprop4}].}
\end{align*}
\item{\eqref{spp3}:} We have
\[
j(\lres{\spp(x)}{x})=j(\spp(x)x)=\spp(x)j(x)=\cov\spp(j(x))j(x)=j(x),
\]
so $\lres{\spp(x)}{x}=x$ because $j$ is monic.
\end{description} 
Finally, taking into account that $\cov\spp$ and $j$ are $B$-equivariant, we have 
\[
\spp(\lres{b}{x})=\cov\spp(j(\lres{b}{x}))=\cov\spp(\lres{b}{j(x)})=b\wedge \cov\spp(j(x))= b\wedge \spp(x)
\]
for all $x\in \opens$ and $b\in B$. This proves that $\spp:\opens\to B$ is equivariant. \qed
\end{proof}

\begin{theorem}\label{funct3, theo: supportopens}
Let $\opens$ be an inverse-embedded quantal frame. Then $\opens$ is a groupoid quantale.
\end{theorem}

\begin{proof}
By Lemma~\ref{lemma:supportofopens} we know that $\opens$ is a supported $B$-$B$-quantal frame, where $B$ is the base locale. Now we can endow $\opens$ with a frame homomorphism $\upsilon:\opens\to B$ defined by $\upsilon(q):=j(q)\wedge e$ for all $q\in \opens$. By \cite{Re07}*{Lemma 3.3, Lemma 3.4}, for all $b\in B$ we have
\[
\upsilon(\lres{b}{1_{\opens}})=j(\lres{b}{1_{\opens}})\wedge e=b1_{\cov\opens}\wedge e=\cov\spp(b)=b=\upsilon(\rres{1_{\opens}}{b}).
\]
This shows that $(\opens, \spp,\upsilon)$ is an equivariantly supported reflexive $B$-$B$-quantal frame.
Let us check that in fact $\opens$ is a groupoid quantale by verifying the following axioms.

\begin{description}

\item[Multiplicativity:] $\cov \opens$ satisfies the multiplicativity axiom because it is an inverse quantal frame~\cite{Re07}*{Cor.\ 4.14}. Since $j\otimes \ident$ is a monomorphism of frames due to Definition~\ref{def:ieqf}\eqref{def:ieq2}, the mapping
\[
\xymatrix{ j\otimes j: \opens\otimes_B\opens\ar@{^{(}->}[rr]^-{j\otimes \ident} && \cov \opens\otimes_B \opens \ar@{^{(}->}[rr]^-{\ident\otimes j} && \cov\opens\otimes_B\cov \opens }
\] 
is a monomorphim, since $\ident\otimes j$ is always a monomorphism because $\cov \opens$ is a flat $B$-module \cite{RR}. By Definition~\ref{def:ieqf}\eqref{def:ieq3}, the diagram
\[
\xymatrix{
\cov \opens\otimes_{B} \cov \opens  & \cov \opens \ar[l]_-{\cov\mu^*}\\
\opens \otimes_{B} \opens \ar@{^{(}->}[u]^{j\otimes j} & \opens \ar[l]^-{\mu^*} \ar@{^{(}->}[u]_{j}
}
\]
commutes, that is, $\cov\mu^*(j(q))=(j\otimes j)\mu^*(q)$ for all $q\in \opens$. So we have
\begin{align*}
(j\otimes j)\bigl(\mu^*\bigl(\V_i x_i\bigr)\bigr) &= \cov\mu^*\bigl(j\bigl(\V_i x_i\bigr)\bigr) =\V_i \mu^*(j(x_i))=\V_i (j\otimes j)(\mu^*(x_i))\\
&=(j\otimes j)\bigl(\V_i\mu^*(x_i)\bigr).
\end{align*}
Finally, since $j\otimes j$ is monic, we conclude that $\mu^*(\V_i x_i)=\V_i\mu^*(x_i)$ and $\opens$ is multiplicative.

\item[Unit laws:] We have to prove that
\begin{equation}\label{eq: unitlawieq}
\V_{xy\leq q} (\lres{\upsilon(x)}{y})=q
\end{equation}
holds for all $x,y,q\in \opens$. Indeed,
\begin{align*}
j(q) &= \V_{\substack{u\in \ipi(\cov \opens) \\ u\leq j(q)}}u\quad = \V_{\substack{u\in \ipi(\cov \opens) \\ u\leq j(q)}}(1_{\cov \opens}\wedge e)u\quad = \V_{\substack{u,w\in \ipi(\cov \opens) \\ u\leq j(q)}}(w\wedge e)u\\
& (\text{because $\cov \opens$ is an inverse quantal frame})\\
&= \V_{\substack{u,w\in \ipi(\cov \opens) \\ u\leq j(q)}}(w\wedge e)w^*u\\
& \text{(because $w\wedge e$ is a subsection of $w^*$)}\\
&= \V_{\substack{u,w\in \ipi(\cov \opens) \\ u\leq j(q)}}(w\wedge e)1_{\cov \opens}\wedge w^*u\quad = \V_{\substack{u,w\in \ipi(\cov \opens) \\ u\leq j(q)}}\cov d^*(\cov\upsilon(w))\wedge w^*u\\
&= \V_{\substack{u,w\in \ipi(\cov \opens) \\ u\leq j(q)}}[\cov d^*,\ident](\cov\upsilon\otimes \ident)(w\otimes w^*u)\ 
=\  [\cov d^*,\ident](\cov\upsilon\otimes \ident)\bigl(\V_{\substack{u,w\in \ipi(\cov \opens) \\ u\leq j(q)}}w\otimes w^*u\bigr)\\
&=\ [\cov d^*,\ident](\cov\upsilon\otimes \ident)(\cov \mu^*(j(q)))\
=\ [\cov d^*,\ident](\cov\upsilon\otimes \ident)((j\otimes j) \mu^*(q))\\
& [\text{due to Definition~\ref{def:ieqf}\eqref{def:ieq3}}]\\
&=\ [\cov d^*,\ident](\cov\upsilon\otimes \ident)(j\otimes j)\bigl(\V_{xy\leq q} x\otimes y\bigr)\\
&=\V_{xy\leq q}[\cov d^*,\ident](\cov\upsilon\otimes \ident)(j(x)\otimes j(y))\\
&=\V_{xy\leq q}\cov d^*(\cov\upsilon(j(x)))\wedge j(y)\\
&=\V_{xy\leq q} \upsilon(x)1_{\cov\opens}\wedge j(y)\\
&=\V_{xy\leq q} \upsilon(x)j(1_{\opens})\wedge j(y)\\
&=\V_{xy\leq q} j(\lres{\upsilon(x)}{1_{\opens}})\wedge j(y)\\
&=\V_{xy\leq q} j(\lres{\upsilon(x)}{y})= j\bigl(\V_{xy\leq q} \lres{\upsilon(x)}{y}\bigr).
\end{align*}

The last four steps hold because $j$ is a frame homomorphism. Thus \eqref{eq: unitlawieq} follows because $j$ is monic.

\item[Inverse laws:] By \cite{QRnonuspp}*{Rem.\ 5.5}, we have to show
\begin{equation}\label{eq:inverselawsieq}
\lres{\upsilon_{\opens}(q)}{1_{\cov \opens}}=\V_{\substack{xx^*\leq q}}x,
\end{equation}
for all $q,x\in \opens$. Indeed,
\begin{align*}
j(\lres{\upsilon(q)}{1_{\opens}})&=\upsilon(q)1_{\cov\opens}\\
&= (j(q)\wedge e)1_{\cov \opens}\\
&\geq \V_{\substack{y\in \cov\opens \\yy^*\leq j(q)}} y&\text{(by \cite{Re07}*{Lemma 4.17})}\\
&\geq \V_{\substack{y\in \opens \\j(y)j(y^*)\leq j(q)}} j(y)\\
&\geq \V_{\substack{y\in \opens \\j(yy^*)\leq j(q)}} j(y)& \text{(by Lemma~\ref{lemma:jpreserveslaxmu})}\\
&= j(\V_{\substack{y\in \opens \\yy^*\leq q}} y).  
\end{align*}
Therefore $\lres{\upsilon(q)}{1_{\opens}}\geq \V_{\substack{y\in \opens \\yy^*\leq q}} y$ because $j$ is monic. The other inequality follows from Definition~\ref{def:ieqf}\eqref{def:ieq4},
\[
j(\lres{\upsilon(q)}{1_{\opens}})=\upsilon(q)1_{\cov\opens}=(j(q)\wedge e)1_{\cov\opens}\leq \V_{\substack{y\in \opens \\yy^*\leq q}} j(y)=j\bigl(\V_{\substack{y\in \opens \\yy^*\leq q}} y\bigr).
\]
Hence, \eqref{eq:inverselawsieq} holds because $j$ is monic. Therefore, $\opens$ is a groupoid quantale as we claimed. \qed
\end{description}
\end{proof}

To close this subsection we shall give an interesting property of the right adjoint of $j$.  

\begin{lemma}\label{lemma:equivarianceofj_*}
Let $\opens$ be an inverse-embedded quantal frame. Then the right adjoint of $j$ is $\ipi(\cov \opens)$-equivariant.
\end{lemma}
\begin{proof}
Since in particular $j$ is a sup-lattice homomorphism, it has a right adjoint $j_*$ given by
\begin{equation}
j_*(x)=\V \{ a\in \opens\ \st\ j(a)\leq x   \}.
\end{equation}
Let us prove that, for all $s\in \ipi(\cov \opens)$ and $x\in \cov\opens$, we have
\begin{equation} \label{eq:equality}
s\cdot j_*(x)=j_*(sx).
\end{equation}
Indeed, 
\begin{align*}
s\cdot j_*(x) &= s\cdot \V \{ a\in \opens\ \st \ j(a)\leq x   \}\\
& = \V \{ s\cdot a\in \opens\ \st \ j(a)\leq x   \}\\
& \leq \V \{ s\cdot a\in \opens\ \st \ sj(a)\leq sx   \}\\
& = \V \{ s\cdot a\in \opens\ \st \ j(s\cdot a)\leq sx   \}& \text{($j$ is $\ipi(\cov \opens)$-equivariant)}\\
& \leq \V \{ q\in \opens\ \st \ j(q)\leq sx   \}\\
&= j_*(sx).
\end{align*}
Moreover, due to stability of the support we have, if $j(q)\leq sx$,
\[
\spp(q)=\cov\spp(j(q))\leq \cov\spp(sx)\leq \cov\spp(s)=ss^*,
\]
and thus
\[
q=(\lres{\spp(q)}{q})\leq (\lres{\cov\spp(s)}{q})\leq (ss^*)\cdot q= (\lres{ss^*}{q})\leq q,
\]
from which it follows that
\[
j_*(sx) = \V \{ (ss^*)\cdot q\in \opens\ \st \ j(q)\leq sx   \}.
\]
Therefore,
\begin{align*}
j_*(sx) &= \V \{ (ss^*)\cdot q\in \opens\ \st \ j(q)\leq sx   \}\\
&=  s\cdot \V \{ s^*\cdot q\in \opens\ \st \ j(q)\leq sx   \}\\
& \leq s\cdot \V \{ s^*\cdot q\in \opens\ \st \ s^*j(q)\leq s^*sx   \}\\
& = s\cdot \V \{ s^*\cdot q\in \opens\ \st \ j(s^*\cdot q)\leq s^*sx   \}& \text{($j$ is an $\cov \opens$-$\cov\opens$-bimodule)}\\
& \leq s\cdot \V \{ a\in \opens\ \st \ j(a)\leq x   \}\\
& \leq s\cdot j_*(x),
\end{align*} 
which proves \eqref{eq:equality}. \qed
\end{proof}

\begin{remark}
For all $u\in \ipi(\cov\opens)$ we have
\begin{align*}
j_*(u)&=j_*(uu^*u)=u\cdot j_*(uu^*)&\text{(due to Lemma~\ref{lemma:equivarianceofj_*})}\\
&\le u\cdot j_*(e)& \text{(because $uu^*\le e$)}\\
&= u\cdot 0&\text{($\opens$ is a non-unital quantale)}\\
&=0. 
\end{align*}
Thus, $j_*(u)=0$ for all $u\in \ipi(\cov\opens)$. 
\end{remark}


\subsection{\'Etale-covered groupoids}

The purpose of this subsection is to establish a bijective correspondence between the class of inverse-embedded quantal frames and a class of open groupoids called \'etale-covered groupoids, that is, open groupoids which admit a suitable notion of covering by an \'etale groupoid.

\begin{definition}\label{cover}
Let $G$ and $H$ be groupoids. We shall say that $H$ \emph{covers} $G$ if there is an epimorphic functor of groupoids $J:H\to G$ such that $J_0:H_0\to G_0$ is an isomorphism. 
\end{definition}

\begin{lemma}\label{extension}
Let $G$ and $H$ be open groupoids such that $H$ covers $G$. Any $G$-action lifts to an $H$-action.
\end{lemma}

\begin{proof}
Let $(X,p,\act)$ be a $G$-action. Notice that the mapping $q:X\to H_0$ defined by $q:=  J_0^{-1}\circ p$ is a map of locales. Let us define $\bct:= \act\circ (J\times \ident_X)$. Diagramatically, we have
\[
\vcenter{\xymatrix{
 H_1\times_{G_0} X  \ar[d]_{J_1\times \ident_X}  \ar[rd]^{\bct} \\
 G_1\times_{G_0} X  \ar[r]_-{\act} & X. }
}
\]
Let us show that $\bct$ is an $H$-action by verifying all the axioms:
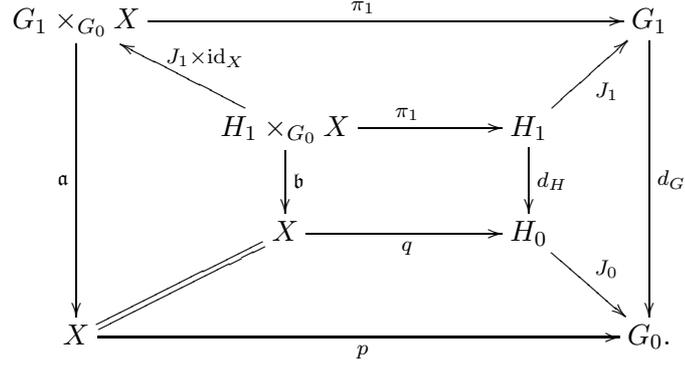
\begin{enumerate} 
\item \textbf{Pullback:}
\[
\vcenter{ \xymatrix{
G_1\times_{G_0} X  \ar[ddd]_{\act} \ar[rrrr]^{\pi_1} &&&& G_1 \ar[ddd]^{d_G} \\ 
& H_1\times_{G_0} X  \ar[ul]_{J_1\times \ident_X} \ar[d]^{\bct} \ar[rr]^{\pi_1} && H_1 \ar[ur]_{J_1} \ar[d]^{d_H}\\
& X \ar@{=}[dl]  \ar[rr]_{q} && H_0 \ar[dr]^{J_0} \\
 X \ar[rrrr]_{p} &&&& G_0. }
}
\]
\item \textbf{Associativity:}
\[
\vcenter{ \small
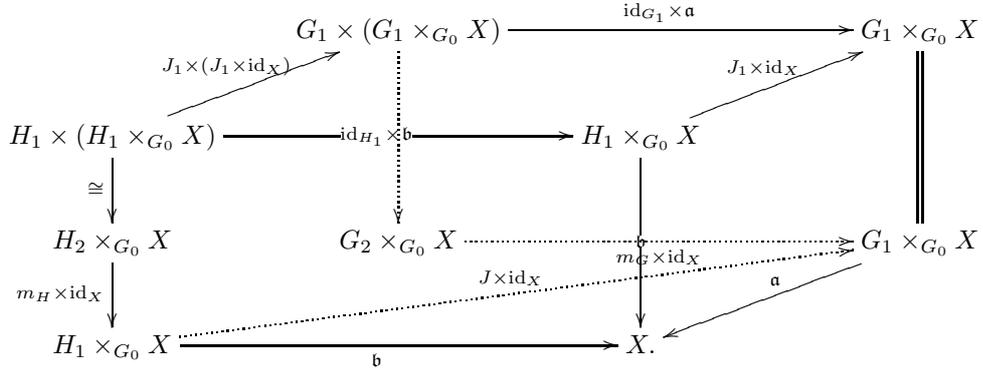
\xymatrix{
& G_1\times (G_1\times_{G_0} X) \ar@{.>}[dd]  \ar[rrr]^{\ident_{G_1}\times \act} &&&  G_1\times_{G_0} X \ar@{=}[dd] \\
H_1\times (H_1\times_{G_0} X) \ar[d]_{\cong} \ar[ur]^-{J_1\times(J_1\times \ident_X)} \ar[rr]| {\ident_{H_1}\times \bct} &&  H_1\times_{G_0} X \ar[urr]^{J_1\times \ident_X} \ar[dd]|{\bct} \\
H_2\times_{G_0} X \ar[d]_{m_H\times \ident_X} & G_2 \times_{G_0} \ar@{.>}[rrr]_-{m_G\times \ident_X} X &&& G_1\times_{G_0} X \ar[dll]_{\act} \\  
H_1\times_{G_0} X \ar@{.>}[urrrr]^{J\times \ident_X} \ar[rr]_{\bct}   && X.}
}
\]
\item\textbf{Unitarity:}
\[
\vcenter
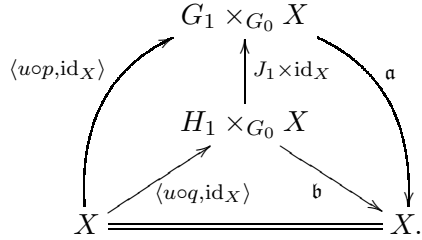
{\xymatrix{
 & G_1\times_{G_0} X \ar@/^2pc/[rdd]^{\act} \\
 & H_1\times_{G_0} X \ar[u]_{J_1\times \ident_X} \ar[dr]_{\bct} \\
 X \ar@/^2pc/[uur]^{\langle u\circ p,\ident_X \rangle} \ar@{=}[rr] \ar[ur]_{\langle u\circ q,\ident_X \rangle} && X.
}
}
\]
\end{enumerate}
Since $\act$ is a $G$-action the above diagrams commute. Therefore it is straightforward to verify that $(X,q,\bct)$ is an $H$-action. \qed
\end{proof}

\begin{lemma}\label{lemma: actionofbisectionsonG}
Let $G$ be an open groupoid and $\cov G$ an \'etale groupoid such that $\cov G$ covers $G$. Then there is an action of $\opens(\cov G)$ on $\opens(G)$. 
\end{lemma}
\begin{proof}
Since  $\ipi(\opens(\cov G))$ is join-dense in  $\opens(\cov G)$, it suffices to show that there is an action of  $\ipi(\opens(\cov G))$ on $\opens(G)$. We begin by identifying $\ipi(\opens(\cov G))$ with $\sections(\cov G)$ because they are isomorphic as involutive monoids (see \cite{PR12}*{Th.\ 3.12}). Let us define $\Phi: \sections(\cov G)\to \sections(G)$ as follows: for every local bisection $s:U\to \cov G_1$ where $U$ is an open sublocale of $\cov G_0$, the mapping $\Phi(s): J_0(U)\to G_1$ (where $J_0(U)$ is the image of $U$ under $J_0$) is given by 
\[
\Phi(s)=J_1\circ s \circ (J'_0)^{-1},
\]
where $J'_0:U\to J_0(U)$ is the restriction of $J_0$ to $U$. We notice that $\Phi(s)$ is in fact a local bisection of $G$: 
\begin{itemize}
\item $d\circ \Phi(s)= d\circ J_1\circ s\circ (J'_0)^{-1}=J_0\circ \cov d\circ  s\circ (J'_0)^{-1}=\ident$ because $s$ is a local bisection of $\cov G$ and $J$ is a functor;
\item $r\circ \Phi(s)= r\circ J_1\circ s\circ (J'_0)^{-1}=J_0\circ \cov r\circ  s\circ (J'_0)^{-1}$ is an open regular monomorphism of locales because $\cov r\circ s$ is, too.
\end{itemize}
In addition, again because $J$ is a functor, we clearly have $\Phi(\cov u) = u$ and $\Phi(s\circ t)=\Phi(s)\circ\Phi(t)$ for all $s,t\in\sections(\cov G)$, so $\Phi$ is a homomorphism of monoids. Since $\opens(G)$ is a groupoid quantale, there is a mapping $\Psi:\sections(G)\to \operatorname{End}(\opens(G))$, that is, there is an action of $\sections(G)$ on $\opens(G)$ (see, \cite{PR12}*{Def.\ 4.10}). The composition $\Psi\circ \Phi: \sections(\cov\opens)\to \opens(G)$ yields an action of $\ipi(\opens(\cov G))$ on $\opens(G)$, which lifts to an action of $\opens(\cov G)$ on $\opens(G)$. \qed  
\end{proof}

\begin{definition}\label{def: etalecoveredgroupoid}
By an \emph{\'etale-covered groupoid} is meant an open groupoid $G$  together with  an \'etale groupoid $\cov G$ such that the following conditions hold.
\begin{enumerate}
\item\label{def: etalecoveredgroupoid1} $\cov G$ covers $G$ by an epimorphic functor $J:\cov G \to G$.
\item\label{def: etalecoveredgroupoid2} $J^*(s\cdot q)=sJ^*(q)$ for all $s\in \ipi(\cov\opens)$ and $q\in \opens(G)$ --- that is, $J^*$ is $\ipi(\cov\opens)$-equivariant --- where the action $\cdot$ is defined as in Lemma~\ref{lemma: actionofbisectionsonG}.
\item\label{def: etalecoveredgroupoid3}  $J_1\times \ident_{G_1}: \cov G_1\times_{G_0} G_1\to G_1\times_{G_0} G_1$ is an epimorphism of locales.
\end{enumerate}
We shall denote the category of $G$-locales and $G$-equivariant maps between them by $G$-$\Loc$. 
\end{definition}

\begin{example}
Every coverable groupoid $G$ in the sense of \cite{PR12} is an \'etale-covered groupoid: $\cov G$ is defined by $\lcc(\sections(\opens(G)))$, and the functor $J:\cov G\to G$ is given by $J^*=j$, where for all $q\in \opens(G)$
\[
j(q)=\V \{ \sigma\in \sections(\opens(G))  \st\ s^*(q)=U \}.
\]
[Recall that $\sigma=(U,s)$ --- \cf\ \eqref{def:bisection}.]
Moreover, in \cite{PR12} it is proved that
\begin{itemize}
\item $j$ and $j\otimes \ident$ are frame monomorphisms,
\item $j(s\cdot q)=sj(q)$ for all $s\in \ipi(\cov\opens)$.
\end{itemize}
Therefore the functor $J$ satisfies all the conditions of Definition~\ref{def: etalecoveredgroupoid}. In particular, every Lie groupoid
is an \'etale-covered groupoid.
\end{example}

\begin{theorem}\label{theorem:protinresende}
If $\opens$ is an inverse-embedded quantal frame then $\mathcal{G}(\opens)$ is an \'etale-covered groupoid. And, conversely, if $G$ is an \'etale-covered groupoid then $\opens(G)$ is an inverse-embedded quantal frame.
\end{theorem}

\begin{proof}
Let us suppose that $\opens$ is an inverse-embedded quantal frame with base locale $B$. By Theorem~\ref{funct3, theo: supportopens} $\opens$ is a groupoid quantale. Let us denote by $G=\mathcal{G}(\opens)$ and $\cov G=\mathcal{G}(\cov\opens)$ the open groupoid and the \'etale groupoid of $\opens$ and $\cov \opens$, respectively. There exists a frame monomorphism $j:\opens \to \cov \opens$ satisfying the conditions of Definition~\ref{def:ieqf}. The latter implies that there exists an epimorphic functor of groupoids $J:\cov G\to G$ such that $J_0$ is the canonical isomorphism with $J^{*}=j$, by Lemma~\ref{lemma:j'isiso}. Clearly, $J_1$ is an epimorphism because $j$ is injective and it is a functor because it satisfies the conditions of \cite{Re15}*{Prop.\ 1.3}, as follows:
\begin{itemize}
\item $J_1\circ \cov i=i\circ J_1$ because $j$ satisfies Definition~\ref{def:ieqf}\eqref{def:ieq1},
\item $J_1\circ \cov u=u$ because $\upsilon=\cov\upsilon\circ j$,
\item $d\circ J_1=\cov d$ because $j$ is a $B$-$B$-bimodule homomorphism,
\item $\cov m\circ (J_1\times J_1)\leq J_1\circ m$ because $j$ satisfies Definition~\ref{def:ieqf}\eqref{def:ieq3}.   
\end{itemize} 
Finally, $J_1\times \ident_{G_1}$ is a surjective map of locales because $j$ satisfies \eqref{def:ieq2}, and $J^*$ is $\ipi(\cov G)$-equivariant because $j$ is an $\cov\opens$-$\cov\opens$-bimodule homomorphism. Therefore, $G$ is an \'etale-covered groupoid. Conversely, let us suppose that $G$ is an \'etale-covered groupoid, and let us denote by $\opens=\opens(G)$ and $Q=\opens(\cov G)$ the groupoid quantale and the inverse quantal frame of $G$ and $\cov G$, respectively. By Lemma~\ref{lemma: actionofbisectionsonG}, $\opens$ is a (left) $Q$-module with action $(s,q)\mapsto s\cdot q$. Notice that the involution of $Q$ makes $\opens$ be a (right) $Q$-module as well, by putting $q\cdot s:=s^*\cdot q$. Therefore $\opens$ is a $Q$-$Q$-bimodule. It is straightforward to see that in fact $\opens$ is an involutive $Q$-$Q$-quantale. Let us verify that $\opens$ verifies the axioms of Definition~\ref{def:ieqf}. In order to do this, notice that $J^*:\opens \to Q$ defines a frame monomorphism which is also a homomorphism of $Q$-$Q$-bimodules. The latter holds because, by assumption, $J^*$ is $\ipi(\opens(\cov G))$-equivariant. Furthermore, it satisfies the various conditions of Definition~\ref{def:ieqf}:
\begin{description}
\item{\eqref{def:ieq1}} because $J^*\circ i=\cov i\circ J^*$, since $J$ is a functor;
\item{\eqref{def:ieq2}} because $J_1\times \ident_{G_1}: \cov G_1\times_{G_0} G_1\to G_1\times_{G_0} G_1$ is a surjective map of locales;
\item{\eqref{def:ieq3}} because $(J^*\otimes J^*)\circ \cov m^*=m^*\circ J^*$,  since $J$ is a functor;
\item{\eqref{def:ieq4}} because $\opens$ is a groupoid quantale, and therefore it satisfies the inverse laws; that is,
\[
\lres{\upsilon(a)}{1_{\opens}}=\V_{\substack{x\in \opens\\ xx^*\le a}} x,
\]
which implies that for all $a\in \opens$ we have
\[
(J^*(a)\wedge e)1_Q=(\cov u^*\circ J^*)(a)1_Q =\V_{\substack{x\in \opens\\ xx^*\le a}} J^*(x),
\]
the latter because $J$ is a functor; 
\item{\eqref{def:ieq0}} because $J_0$ is an isomorphism. 
\end{description} 
This proves that $\opens$ is an inverse-embedded quantal frame. \qed
\end{proof}


\section{Actions}\label{section:actionsetalecoveredgprds}

Let us now see that an appropriate notion of action for inverse-embedded quantal frames yields an equivalence of categories $G$-$\Loc\cong \opens$-$\Loc$ where $\opens=\opens(G)$ is the quantale of an \'etale-covered groupoid $G$. Then we obtain the two main applications of this paper, namely (i)~a definition of sheaf for inverse-embedded quantal frames that extends that of \cite{GSQS} for \'etale groupoids and completely characterizes the sheaves of \'etale-covered groupoids; and (ii)~an extension of the functoriality results of \cite{Re15}.


\subsection{Descent actions}
 
\begin{lemma}
Let $G$ be an \'etale-covered groupoid. Then any $G$-locale $X$ is an $\opens(\cov{G})$-locale.  
\end{lemma}

\begin{proof}
By Lemma~\ref{extension} any $G$-action $X$ can be extended to a $\cov{G}$-action. Then, taking into account that $\cov G$-$\Loc$ is isomorphic to $\opens(\cov G)$-$\Loc$ (\cf\ \cite{GSQS}*{Lemma 3.8, Lemma 3.19}), we conclude that $X$ is also an $\opens(\cov{G})$-locale. \qed
\end{proof}

In the context of rings, the notion of descent theory of modules can be described briefly. When $R$ and $S$ are two commutative rings with unit connected by a homomorphism $f:R\to S$, there is an obvious way of viewing an $S$-module as an $R$-module. Moreover, given an $R$-module $N$, there is a natural associated $S$-module, $N\otimes_R S$, which is called the $S$-module induced by $N$. This association of $S$-modules with $R$-modules is an expression of an adjointness relation which is a common topic to many algebraic constructions. The latter inspired us to define the following notion of descent for \'etale-covered groupoids:

\begin{definition}\label{descend}
Let $G$ be an \'etale-covered groupoid. A $\cov{G}$-action $(X,p,\cov{\act})$ satisfies the \emph{descent condition} if there exists a map of locales $\bct: G_1\times_{G_0} X\to X$ such that the following diagram:
\[
\xymatrix{
\cov{G}_1\times_{G_0} X \ar[rr]^-{\cov{\act}}  \ar@{->>}_{J_1\times \ident_X}[d] 
&& X \\
G_1\times_{G_0} X\ar[urr]_-{\bct}  }
\]
commutes in \Loc.
\end{definition}

\begin{lemma}\label{lemmadescend}
Let $G$ be an \'etale-covered groupoid, and let $(X,p,\cov{\act})$ be a $\cov G$-action that satisfies the descent condition. Then $(X,p,\bct)$ is a $G$-action.
\end{lemma}

\begin{proof}
Let $(X,p,\cov{\act})$ be a $\cov{G}$-action that satisfies the descent condition. Let us show that in fact $(X, p, \bct)$ is a $G$-action by verifying all the axioms, similarly to what we did in Lemma~\ref{extension}:
\begin{enumerate} 
\item \textbf{Pullback:}
\[
\vcenter{ \xymatrix{
\cov{G}_1\times_{G_0} X \ar[dr]_{J_1\times \ident_X}  \ar[ddd]_{\act} \ar[rrrr]^-{\cov{\pi}_1} &&&& \cov{G}_1 \ar[dl]_{J} \ar[ddd]^{\cov{d}} \\ 
& G_1\times_{G_0} X   \ar[d]^{\bct} \ar[rr]^-{\pi_1} && G_1  \ar[d]^-{d}\\
& X \ar@{=}[dl]  \ar[rr]_-{p} && G_0 \ar@{=}[dr] \\
 X \ar[rrrr]_{p} &&&& G_0. }
}
\]
\item \textbf{Associativity:}
\[
\vcenter{ \small\xymatrix{
& \cov{G}_1\times (\cov{G}_1\times_{G_0} X) \ar[dl]_-{J_1\times(J_1\times \ident_X)} \ar@{.>}[dd]  \ar[rrr]_-{\ident_{\cov{G}_1}\times \act} &&&  \cov{G}_1\times_{G_0} X \ar@{=}[dd]  \ar[dll]^{J_1\times \ident_X}\\
G_1\times (G_1\times_{G_0} X) \ar[d]_{\cong}  \ar[rr]^-{\ident_{G_1}\times \bct} &&  G_1\times_{G_0} X  \ar[dd]^{\bct} \\
G_2\times_{G_0} X \ar[d]_{m_G\times \ident_X} & \cov{G}_2 \times_{G_0} \ar@{.>}[rrr]^-{m_{\cov{G}}\times \ident_X} X &&& \cov{G}_1\times_{G_0} X \ar[dll]_{\act} \\  
G_1\times_{G_0} X \ar@{.>}[urrrr]^{J\times \ident_X} \ar[rr]_{\bct}   && X.}
}
\]
\item\textbf{Unitarity:}
\[
\vcenter{\xymatrix{
 & \cov{G}_1\times_{G_0} X \ar[d]_{J_1\times \ident_X} \ar@/^2pc/[rdd]^{\act} \\
 & G_1\times_{G_0} X  \ar[dr]_{\bct} \\
 X \ar@/^2pc/[uur]^{\langle u\circ p,\ident_X \rangle} \ar@{=}[rr] \ar[ur]_{\langle u\circ q,\ident_X \rangle} && X.
}
}
\]
\end{enumerate}
Since $\cov{\act}$ is a $\cov{G}$-action such that $\cov \act = \bct\circ (J_1\times \ident_X)$ due to the descent condition, all the above diagrams commute. Therefore $(X,q,\bct)$ is a $G$-action. \qed
\end{proof}


\subsection{Categories of actions}\label{funct3, subsec: open-loccat}

\begin{definition}\label{oq-loc}
Let $\opens$ be an inverse-embedded quantal frame. By an \emph{$\opens$-locale} is meant an $\cov \opens$-locale  $X$ (with action $\alpha$) such that $\alpha_*$ factors (necessarily uniquely) through $j\otimes\ident_X$ in \Frm. The category $\opens$-\Loc\ consists of $\opens$-locales as objects, and the morphisms are the maps of locales whose inverse images are homomorphisms of left $\cov \opens$-modules.
\end{definition}

Clearly, the category $\opens$-\Loc\ is a full subcategory of $\cov \opens$-\Loc.

\begin{example}
Any inverse-embedded quantal frame $\opens$ is an $\opens$-locale. Indeed, let us write $f:\cov \opens\otimes \opens\to \opens$ for the action of $\cov\opens$ on $\opens$. By \cite{GSQS}*{Lemma 3.15} its right adjoint can be written as 
\[
f_*(x)=\V_{s\in \ipi(\cov \opens)}s\otimes s^*\cdot x,
\]  
and it is a frame homomorphism. Moreover, for all $x\in \opens$ we have
\begin{align*}
(\ident_{\cov\opens}\otimes j)\circ f_*(x) &= (\ident_{\cov\opens}\otimes j)(\V_{s\in \ipi(\cov \opens)}s\otimes s^*\cdot x)\\
&= \V_{s\in \ipi(\cov \opens)}s\otimes j(s^*\cdot x)\\
&= \V_{s\in \ipi(\cov \opens)}s\otimes s^*j(x) \quad\quad \begin{minipage}{2cm}\text{($j$ is an $\cov\opens$-$\cov\opens$-bimodule}\\
\text{homomorphism)}\end{minipage}\\
&= \cov\mu^*(j(x)),
\end{align*} 
where the last equality follows from \eqref{funct3, eq: multiplicativityopens}. Finally, let us prove that $f_*$ factors (uniquely) through $j\otimes \ident_{\opens}$ in $\Frm$. In order to do this, let us notice that
\begin{align*} 
(\ident_{\cov\opens}\otimes j)\circ(j\otimes \ident_{\opens})\circ\mu^*&=(j\otimes j)\circ \mu^*\\
&= \cov\mu^*\circ j& \text{[by Definition~\ref{def:ieqf}\eqref{def:ieq3}]}\\
&=(\ident_{\cov\opens}\otimes j)\circ f_*.
\end{align*}
Then $(j\otimes \ident_{\opens})\circ\mu^*=f_*$ because $\ident_{\cov\opens}\otimes j$ is monic. Therefore $\opens$ is an $\opens$-locale. 
\end{example}

\begin{lemma}\label{lemma:bijection, qoloc}
Let $G$ be an \'etale-covered groupoid. The assignment $X\mapsto \opens(X)$ from $G$-$\Loc$ to $\opens(G)$-$\Loc$ is a bijection up to isomorphisms.
\end{lemma}

\begin{proof}
Let $\cov \opens=\opens(\cov{G})$ and $\opens=\opens(G)$. Let $(X,p,\cov{\act})$ be a $\cov{G}$-action satisfying the descent condition. Then there exists a map of locales $\bct: G_1\times_{G_0} X\to X$ such that the following diagram is commutative:
\[
\xymatrix{
\cov{G}_1\times_{G_0} X \ar[rr]^-{\cov{\act}}  \ar@{->>}_{J_1\times \ident_X}[d] 
&& X \\
G_1\times_{G_0} X.\ar[urr]_-{\bct} }
\]
Moreover $(X,p,\bct)$ is a $G$-action due to Lemma~\ref{lemmadescend}. Now, if we consider the inverse image of the above maps, we obtain the following commutative diagram in $\Frm$, where $B$ is the base locale:
\[
\xymatrix{
\cov \opens \otimes_B X   
&& X \ar[ll]_-{\cov{\act}^*}\ar[lld]^-{{\bct}^*} \\
\opens \otimes_B X. \ar@{^{(}->}^{j\otimes \ident_X}[u]  }
\]
This, in terms of frames, means that the action $\alpha: \cov \opens\otimes_{B} X\to X$ of $\cov \opens$ on $X$, which has a join preserving right adjoint $\alpha_*:\cov \opens\otimes_{B} X\to X$  (this is the inverse image $\cov{\act}^{*}$ of the action $\cov{\act}: \cov{G}_1\times_{G_0} X\to X$), factors (necessarily uniquely) through $j\otimes\ident_X$ in $\Frm$. Hence, $X$ is an $\opens$-locale. Conversely, let $X$ be an $\opens$-locale. Then there exists a frame homomorphism $\beta_*:\opens\otimes_B X\to X$ such that the following diagram commutes:
\[
\xymatrix{
\cov\opens\otimes_B X  && X \ar[ll]_-{\alpha_*}\ar[lld]^-{\beta_*} \\
\opens \otimes_B X. \ar@{^{(}->}^{j\otimes \ident_X}[u]  }
\]
In terms of locales this means that there exists a map of locales $\bct: G_1\times_{G_0} X\to X$ such that the diagram
\[
\xymatrix{
\cov{G}_1\times_{G_0} X \ar[rr]^-{\cov{\act}}  \ar@{->>}_{J_1\times \ident_X}[d] 
&& X  \\
G_1\times_{G_0} X\ar[urr]_-{\bct}  }
\]
is commutative in $\Loc$. So $(X,p,\cov{\act})$ satisfies the descent condition. \qed 
\end{proof}

\begin{remark}\label{remark:opensmodules}
Note that $\opens$-locales are also $\opens$-modules in the usual sense. This is because, following the proof of Lemma~\ref{lemma:bijection, qoloc}, the factorization of the inverse image $\alpha_*$ of the action $\alpha:\cov\opens\otimes_B X\to X$ is done via a frame homomorphism $\beta:X\to\opens\otimes_B X$ which turns out to be the inverse image of a groupoid action $\bct: G_1 \times_{G_0} X\to X$, and thus the left adjoint of $\beta$ is $\bct_!$. Hence, $X$ is an $\opens$-module.
\end{remark}

\begin{lemma}\label{lemma:j(aleqax)}
Let $\opens$ be an inverse-embedded quantal frame with base locale $B$, and let $X$ be an $\opens$-locale. Then, for all $a\in \opens$ and $x\in X$, we have
\[
j(a)x\le ax.
\]
\end{lemma}

\begin{proof}
By Remark~\ref{remark:opensmodules}, we know that $X$ is an $\opens$-module with action $\beta:\opens\otimes_B X\to X$. Therefore,
\begin{align*}
j(a)x &= \alpha(j(a)\otimes x) = \alpha((j\otimes \ident_X)(a\otimes x))& \text{($\alpha$ is the action of $\cov\opens $ on $X$)}\\
&\leq \alpha((j\otimes \ident_X)(\beta_*(\beta(a\otimes x)))& \text{(because $\beta_*\circ \beta \geq \ident$)}\\
&= \alpha\circ \alpha_*\circ \beta_!(a\otimes x)&\text{(by the descent condition)}\\
&\leq \beta(a\otimes x)=ax & \text{(because $\alpha\circ \alpha_*\leq \ident$)}. \qed
\end{align*}
\end{proof}

\begin{theorem}\label{funct3, theo: gloccongqoloc}
Let $G$ be an \'etale-covered groupoid. Then the categories $G$-$\Loc$ and $\opens$-$\Loc$ are equivalent.
\end{theorem}

\begin{proof}
Let $G$ be an \'etale-covered groupoid. The assignment $X\mapsto \opens(X)$ from $G$-$\Loc$ to $\opens$-$\Loc$ is a bijection due to Lemma~\ref{lemma:bijection, qoloc}. Now let $X$ and $Y$ be arbitrary $G$-locales, and let $f:X\to Y$ be a $G$-equivariant map. Let us prove that $f^*:\opens(Y)\to \opens(X)$ is a homomorphism of $\cov \opens$-locales. Recall that the categories $\cov{G}$-$\Loc$ and $\cov \opens$-$\Loc$ are equivalent \cite{GSQS}*{Lemma 3.19}. Then it suffices to show that $f$ is a $\cov{G}$-equivariant map. Indeed, let us consider the following diagram:
\[
\xymatrix{
\cov{G}_1\times_{G_0}X \ar^{\ident_{\cov{G}_1}\times f}[rr] \ar^{J_1\times \ident_X}[d] \ar@/_3pc/[dd]_{\cov{\act}} && \cov{G}_1\times_{G_0}Y  \ar_{J_1\times \ident_Y}[d]\ar@/^3pc/[dd]^{\cov{\bct}} \\
G_1\times_{G_0}X \ar_{\act}[d] \ar_{\ident_{G_1}\times f}[rr]&& G_1\times_{G_0}Y \ar^{\bct}[d]\\
X \ar_{p}[dr] \ar^{f}[rr] && Y\ar^{q}[dl]\\
& G_0.
}
\]
where $\cov \act = \act\circ (J_1\times \ident_X)$ and $\cov \bct =  \bct\circ (J_1\times \ident_X)$. Clearly, $f$ commutes with the actions $\cov{\act}$ and $\cov{\bct}$ because $f$ commutes with $\act$ and $\bct$:
\begin{align*}
f\circ \cov \act &= f\circ \act \circ (J_1\times \ident_X)\\
&= \bct\circ (\ident_{G_1}\times f)\circ (J_1\times \ident_X)\\
&= \bct\circ (J_1\times \ident_Y)\circ (\ident_{\cov G_1}\times f)\\
&= \cov \bct \circ (\ident_{\cov  G_1}\times f).
\end{align*}
Moreover, $p=q\circ f$. So, we conclude that $f^*$ is a morphism in $\cov \opens$-$\Loc$. Finally, let $\opens(X)$ and $\opens(Y)$ be arbitrary $\opens$-locales, and let $f^*:\opens(Y)\to \opens(X)$ be a map of $\cov \opens$-locales. We want to show that $f$ is a $G$-equivariant map. In fact, since $\cov{G}$-$\Loc$ and $\cov \opens$-$\Loc$ are equivalent, $f:(X,\cov{\act},p)\to(X,\cov{\bct},p)$ is a $\cov{G}$-equivariant map. Now, taking into account that both $(X,\cov{\act},p)$ and $(X,\cov{\bct},q)$ satisfy the descent condition, we have
\[
\xymatrix{
& {\cov{G}_1}\times_{G_0}X  \ar_{J_1\times \ident_X}[dl] \ar^{\ident_{{\cov{G}_1}}\times f}[rr]  \ar_{\cov{\act}}[dd]  && {\cov{G}_1}\times_{G_0}Y \ar^{J_1\times \ident_Y}[dr]  \ar_{\cov{\bct}}[dd]  \\
G_1\times_{G_0}X \ar@{.>}[dr]_{\exists \act} &&&& G_1\times_{G_0}Y \ar@{.>}[dl]^{\exists \bct} \\
& X\ar^{f}[rr] \ar_{p}[dr] && Y\ar^{q}[dl]\\
&& G_0.
}
\]  
Clearly, $p=q\circ f$. And, since $f\circ \cov{\act}=\cov{\bct}\circ (\ident_{{\cov{G}_1}}\times f )$, we have
\begin{align*}
(f\circ \act)\circ (J_1\times \ident_X) & = \bct\circ ((J_1\times \ident_Y)\circ (\ident_{{\cov{G}_1}}\times f )).
\end{align*}
Furthermore, the diagram
\[
\xymatrix{
 {\cov{G}_1}\times_{G_0}X  \ar^{J_1\times f}[drr] \ar_{J\times \ident_X}[d] \ar^{\ident_{{\cov{G}_1}}\times f}[rr] &&  {\cov{G}_1}\times_{G_0}Y  \ar^{J_1\times \ident_Y}[d]  \\
G_1\times_{G_0}X \ar_{\ident_{G_1}\times f}[rr] && G_1\times_{G_0}Y.}
\] 
is commutative, and thus
\begin{align*}
\bct\circ ((J_1\times \ident_Y)\circ (\ident_{{\cov{G}_1}}\times f )) & = \bct \circ (J_1\times f)\\
& = \bct \circ ((\ident_{G_1}\times f)\circ(J_1\times \ident_X)).
\end{align*}
Hence, 
\[
(f\circ \act)\circ (J_1\times \ident_X)= (\bct \circ (\ident_{G_1}\times f))\circ(J_1\times \ident_X).
\]
Finally, since $J_1\times \ident_X$ is an epimorphism due to Definition~\ref{def: etalecoveredgroupoid}\eqref{def: etalecoveredgroupoid3}, we can conclude
\[
f\circ \act= \bct \circ (\ident_{G_1}\times f). \qed
\]
\end{proof}

\begin{remark}
If $G$ is an \'etale-covered groupoid then any map $f:X\to Y$ of $G$-locales is $G$-equivariant if and only if it is $\cov G$-equivariant. This in turn is equivalent to $f^*$ being $\opens(\cov G)$-equivariant, by the results for \'etale groupoids. But $f$ being $G$-equivariant also implies that $f^*$ is $\opens(G)$-equivariant, that is, a homomorphism of $\opens(G)$-modules, by \cite{GSQS}*{Lemma 3.6}.
\end{remark}


\subsection{Orbits}

Recall that if $G$ is an open groupoid and $X$ is a $G$-locale, the \emph{orbit locale} $X/G$ can be constructed as the following coequalizer in $\Loc$:
\begin{equation}
\xymatrix{
G_1\times_{G_0} X\ar@<0.7ex>[rr]^-{\act}\ar@<-0.7ex>[rr]_-{\pi_2}&& X \ar@{->>}[r]^-{\pi}& X/G.
}
\end{equation}

\begin{definition}
Let $G$ be an \'etale-covered groupoid with inverse-embedded quantal frame $\opens=\opens(G)$, and let $X$ be a (left) $G$-locale. An element $x\in X$ is \emph{invariant} if the following equivalent conditions hold (regarding $X$ as an $\opens$-module):
\begin{enumerate}
\item For all $q\in \opens$ we have $qx\leq x$;
\item $1_{\opens}x\leq x$;
\item $1_{\opens}x=x$.
\end{enumerate}
The set of invariant opens of $X$ will be denoted by $I_{\opens}(X)$.
\end{definition}

\begin{lemma}\label{lemma:IO(X)subsetIQ(X)}
Let $G$ be an \'etale-covered groupoid, and let $X$ be a left $G$-locale. Then $I_{\opens}(X)\subset I_{\cov\opens}(X)$.
\end{lemma}
\begin{proof}
Notice that for all $x\in I_{\opens}(X)$ we have
\begin{align*}
1_{\cov\opens}\cdot x &= j(1_{\opens})\cdot x\le 1_{\opens}x\le x& \text{(by Lemma~\ref{lemma:j(aleqax)}).}
\end{align*} 
This implies that $x\in  I_{\cov\opens}(X)$. \qed
\end{proof}

For \'etale-covered groupoids we have a simple description of these quotients in terms of quantale modules.

\begin{theorem}\label{theorem:invariantelements}
Let $G$ be an \'etale-covered groupoid, and let $X$ be a left $G$-locale. The quotient $X/G$ coincides with the set of invariant elements of the action. Moreover, $I_{\cov\opens}(X)=I_{\opens}(X)$. 
\end{theorem}

\begin{proof}
Since $(X,\bct)$ is a $G$-locale,  due to Theorem~\ref{funct3, theo: gloccongqoloc} there exists a $\cov G$-action $\act$ such that $(j\otimes \ident_X)\circ \bct^*=\act^*$ in $\Frm$.
Let us prove that the following diagram is an equalizer in \sets, where $\iota$ is the frame inclusion:

\begin{equation*}
\xymatrix{
I_{\opens}(X) \ar@{^{(}->}[r]^-{\iota} & X \ar@<0.7ex>[rr]^-{\bct^*}\ar@<-0.7ex>[rr]_-{\pi^*_2}&& \opens\otimes_{B}X.
}
\end{equation*}
In other words, we need to show that $x\in I_{\opens}(X)$ is invariant if and only if 
\begin{equation}\label{eq:pi2equalsbct}
\pi^*_2(x)=\bct^*(x).
\end{equation}
If \eqref{eq:pi2equalsbct} holds, we have
\begin{align*}
1_{\opens}x &= \bct_!(1_{\opens}\otimes x) =\bct_!(\pi^*_2(x)) =\bct_!(\bct^*(x))\le x& \text{($\bct_!\circ\bct^*\le \ident$),} 
\end{align*}
so $x\in I_{\opens}(X)$. Conversely, the condition
$1_{\opens}x\le x$ implies
\[
1_{\opens}\otimes x\le \bct^*(x)=\V_{qy\le x} q\otimes y.
\]
And, because $(j\otimes \ident_X)\circ \bct^*=\act^*$, we have
\begin{align*}
(j\otimes \ident_X)\circ \bct^*(x)=\act^*(x)&=\V_{u\in \ipi(\cov\opens)} u\otimes u^*x\\
&\le \V_{u\in \ipi(\cov\opens)} u\otimes x\quad\quad \begin{minipage}{2cm}\text{($x\in  I_{\cov\opens}(X)$ due}\\ \text{to Lemma~\ref{lemma:IO(X)subsetIQ(X)})}\end{minipage}\\
&=1_{\cov\opens}\otimes x=j(1_{\opens})\otimes x = (j\otimes \ident_X)(\pi^*_2(x)).
\end{align*}
Hence,  $\bct^*=\pi^*_2$ because $j\otimes \ident_X$ is monic, and thus \eqref{eq:pi2equalsbct} holds. Finally, if $x\in I_{\cov\opens}(X)$ we have
\begin{align*}
(j\otimes \ident_X)\circ \bct^*(x) &=\act^*(x)=\cov\pi^*_2(x)&\text{($x\in I_{\cov\opens}(X)$)}\\
&= (j\otimes \ident_X)\circ \pi^*_2(x).
\end{align*}
Hence, $\bct^*(x)= \pi^*_2(x)$ because $j\otimes \ident_X$ is monic, so we conclude that $I_{\cov\opens}(X)\subset I_{\opens}(X)$. And $I_{\opens}(X)\subset I_{\cov\opens}(X)$ holds by Lemma~\ref{lemma:IO(X)subsetIQ(X)}. \qed
\end{proof}

From now on, given an \'etale-covered groupoid $G$, we shall denote the set of invariant elements of a left $G$-locale $X$ by $I(X)$.


\subsection{Sheaves}\label{funct3, sec: openssheaves}

Recall that a $G$-action whose anchor map is a local homeomorphism is called a \emph{sheaf} over $G$ or simply a $G$-sheaf. Furthermore, if $G$ is an \'etale groupoid then $BG$ (the topos of equivariant sheaves on $G$) is equivalent to the category of sheaves on the involutive quantale $\opens(G)$ of the groupoid \cite{GSQS}. 
The purpose of this section is to generalize the latter. Indeed we shall prove that the topos $BG$ of an \'etale-covered groupoid $G$ is isomorphic to the category of $\opens(G)$-sheaves.   

\begin{definition}\label{oq-sheaf}
Let $\opens$ be an inverse-embedded quantal frame. By an \emph{$\opens$-sheaf} $X$ will be meant an $\cov \opens$-sheaf $X$ (equivalently, a complete Hilbert $\cov\opens$-module) such that for all $x,y\in X$ we have
\[
\langle x,y \rangle\in j(\opens).
\]
The \emph{category of $\opens$-sheaves} is denoted by $\opens$-\textit{Sh}. It has the $\opens$-sheaves as objects, and its morphisms are the sheaf homomorphisms. Note that $\opens$-\textit{Sh} is a full subcategory of $\cov\opens$-\textit{Sh}. 
\end{definition}

\begin{theorem}\label{qo-sheaves}
Let $G$ be an \'etale-covered groupoid, and let $X$ be a $\cov{G}$-sheaf. Then $X$ satisfies the descent condition if and only if the inner product induced by $X$ is valued in $j(\opens(G))$. 
\end{theorem}

\begin{proof}
Let us write $\opens=\opens(G)$ and $\cov\opens=\opens(\cov{G})$ for the quantales of the \'etale-covered groupoid $G$ and its \'etale cover $\cov G$, respectively, and $B$ for the base locale. Let $(X,\act,p)$ be a $\cov G$-sheaf that satisfies the descent condition. This is equivalent to saying that there exists a $G$-action $\bct$ such that $(j\otimes \ident_X)\circ \bct^*=\act^*$ in $\Frm$, due to Lemma~\ref{lemma:bijection, qoloc}. Let us consider the following coequalizer in $\Loc$:
\[
\xymatrix{
\cov{G_1}\times_{G_0} X \ar@<0.5ex>[r]^-{\act}\ar@<-0.5ex>[r]_-{\cov\pi_2} & X\ar@{->>}[r]^-{\pi}   & X/\cov{G}.
}
\]
Now consider the pullback  $X\times_{X/\cov{G}} X$ of the quotient map $\pi$ and the pairing map $\langle \act,\cov\pi_2\rangle$ from $\cov{G_1}\times_{G_0} X$ to $X\times_{X/\cov{G}} X$, whose inverse image homomorphism is valued in $j(\opens)\otimes_B X$, as the following derivation with $x,y\in X$ shows:
\begin{align*} 
[\act^*,\cov \pi^*_2](x\otimes y) &=\act^*(x)\wedge \cov\pi^*_2(y)=(j\otimes \ident_X)(\bct^*(x))\wedge (j\otimes \ident_X)(\pi^*_2(y)) \\
&=(j\otimes \ident_X)\circ[\bct^*,\pi^*_2](x\otimes y).  
\end{align*}
Therefore we can define the composition
\[
\xymatrix{
X\otimes_{I(X)} X \ar[rr]^-{[\act^*,\cov\pi^*_2]}&& j(\opens)\otimes_{B} X\ar[r]^-{(\cov\pi_1)_!}   & j(\opens),
}
\]
where $(\cov\pi_1)_!$, given by $j(q)\otimes x\mapsto j(q)\spp_X(x)$, is the direct image of the map of locales $\cov\pi_1:\cov G_1\times_{G_0} X\to \cov G_1$, due to \cite{Funct2}*{Lemma 3.1} (with the only difference that, contrary to the situation in that lemma, the action of $B$ on $j(\opens)$ is on the right).
Hence, for all $x,y\in X$, we have
\begin{align*}
\langle x,y\rangle &=\V_{u\in \ipi(\cov\opens)}u\spp_X(u^*x\wedge y)& \text{(by Theorem~\ref{theorem:innerproductformula})}\\
&=(\cov\pi_1)_!(\V_{u\in \ipi(\cov\opens)} u\otimes (u^*x\wedge y))\\
&=(\cov\pi_1)_!((\V_{u\in \ipi(\cov\opens)} u\otimes u^*x)\wedge (1_{\cov\opens}\otimes y))\\
&=(\cov\pi_1)_!(\act^*(x)\wedge \cov\pi^*_2(y))&\text{[by \eqref{eq:rightadjformula2}]}\\
&=(\cov\pi_1)_!([\act^*,\cov\pi^*_2](x\otimes y))\\
&\in j(\opens).
\end{align*}
Conversely, let us suppose that the inner product of $X$ is valued in $j(\opens)$. Then, for all $s,t\in\sections_X$,
\begin{align*}
[\act^*,\cov\pi^*_2](s\otimes t) &= \V_{u\in \ipi(\cov\opens)}u\otimes (u^*s\wedge t)\\
&= \V_{u\in \ipi(\cov\opens)}u\otimes \spp_X(u^*s\wedge t)t& (\text{$u^*s\wedge t$ is a subsection of $t$})\\
&= \V_{u\in \ipi(\cov\opens)}u\spp_X(u^*s\wedge t)\otimes t& (\text{$\otimes$ is over $B$})\\
&= \langle s,t\rangle \otimes t& (\text{by Theorem~\ref{theorem:innerproductformula}})\\ 
&\in  j(\opens)\otimes_{B} X& (\text{by assumption}).
\end{align*} 
Then, since $\lcc(\sections_X)$ is join-dense in $X$, we conclude that 
\[
[\act^*,\cov\pi^*_2](x\otimes y)\in j(\opens)\otimes_{B} X\ \text{for all $x,y\in X$.}
\]
This implies that $[a^*,\cov\pi^*_2]$ factors (uniquely) through a frame homomorphism $\phi$ such that the following diagram commutes in $\Frm$:
\[
\xymatrix{
X\otimes_{I(X)} X\ar[r]^-{[\act^*,\cov\pi^*_2]}\ar[dr]_-{\phi}  & \cov\opens\otimes_B X\\
& \opens\otimes_B X \ar@{^{(}->}[u]_-{j\otimes \ident_X}.
}
\]
Then the frame homomorphism given by the composition 
\[
\xymatrix{
X \ar[r]^-{\pi^*_1}& X\otimes_{I(X)} X\ar[rr]^-{[\act^*,\cov\pi^*_2]}   && j(\opens)\otimes_{B} X
}
\]
is such that 
\begin{align*} 
(j\otimes \ident_X)\circ (\phi\circ \pi^*_1)(x) &=((j\otimes \ident_X)\circ \phi)\circ\pi^*_1(x)\\ 
&=[\act^*,\cov\pi^*_2]\circ\pi^*_1(x)=[\act^*,\cov\pi^*_2](x\otimes 1_X)\\
&=\act^*(x)\wedge \cov\pi^*_2(1_X)=\act^*(x)\wedge (1_{\cov\opens}\otimes 1_X)=\act^*(x), 
\end{align*}
which implies that $X$ satisfies the descent condition with $\bct^* = \phi\circ\pi_1^*$. \qed
\end{proof}

\begin{remark}
Let $G$ be an \'etale-covered groupoid. Any principally covered $\cov{G}$-sheaf $X$ is such that $\langle X,X\rangle\in \ipi(\opens(\cov{G}))$~\cite{Funct2}*{Lemma 5.3}. Therefore the principally covered $\cov{G}$-sheaves provide an example of a class of $\cov{G}$-actions which does not satisfy the descent condition.
\end{remark}

\begin{corollary}
Let $G$ be an \'etale-covered groupoid. Then $BG$ is equivalent to $\opens(G)$-\textit{Sh}.
\end{corollary}

\begin{proof}
This follows from Theorem~\ref{qo-sheaves} and \cite{GSQS}*{Th.\ 4.62}. \qed
\end{proof}


\subsection{Bi-actions and functoriality}


\paragraph{Groupoid bilocales.}

Now we address the second aim of this paper, which is to show that the bilocales of \'etale-covered groupoids can be identified with a natural notion of bilocale for inverse-embedded quantal frames, and from this to establish a (bicategorical) equivalence that generalizes that of~\cite{Re15} between \'etale groupoids and inverse quantal frames. 

\begin{definition}\label{biactions}
Let $G$ and $H$ be open localic groupoids. A \emph{$G$-$H$-bilocale} is a locale $_GX_H$ (often denoted only by $X$), equipped with a left $G$-locale structure $(p,\act)$ and a right $H$-locale structure $(q,\bct)$ such that the following diagrams commute in $\Loc$.
\begin{enumerate}
\item $q$ is invariant
under the action of $G$:
\[
\vcenter{\xymatrix{
G_1\times_{G_0} X \ar[rr]^-{\act} \ar[d]_{\pi_2} && X \ar[d]_{q} \\
X \ar[rr]_{q}  && H_0.   }
}
\]
\item $p$ is invariant
under the action of $H$:
\[
\vcenter{\xymatrix{
X\times_{H_0} H_1 \ar[rr]^-{\bct} \ar[d]_{\pi_1} && X \ar[d]_{p} \\
X \ar[rr]_{p}  && G_0.   }
}
\]
\item Associativity:
\[
\vcenter{\xymatrix{
G_1\times_{G_0} X\times_{H_0} H_1 \ar[rr]^-{\act\times \ident_{H_1}} \ar[d]_{\ident_{G_1}\times \bct} && X\times_{H_0} H_1 \ar[d]_{\bct} \\
G_1\times_{G_0} X \ar[rr]_{\act}  && X.   }
}
\]
\end{enumerate}
A map of bilocales $f: {_G X_H}\to  {_G Y_H}$ is a map of locales which is both a map of left $G$-locales and a map of right $H$-locales. We will denote the \emph{category of $G$-$H$-bilocales} by $G$-$H$-$\Loc$.
\end{definition}

Based on the previous definition we can define the bicategory on which the rest of this paper will be based:

\begin{definition}
The \emph{bicategory of \'etale-covered groupoids $\biGrpd$} is defined as follows.
\begin{itemize}
\item The $0$-cells are the \'etale-covered groupoids.
\item The $1$-cells $X:G\to H$ are the $G$-$H$-locales.
\item The composition of $1$-cells is defined by tensor product --- given $1$-cells $X:G\to H$ and $Y:H\to K$ we define $Y\circ X:= X\otimes_H Y$ to be the coequalizer in $\Loc$
\begin{equation}
\xymatrix{
X\times_{G_0} H_1 \times_{G_0} Y\ar@<0.7ex>[rr]^-{\langle  \act\circ \pi_{12},\pi_3 \rangle}\ar@<-0.7ex>[rr]_-{\langle \pi_1,\bct\circ \pi_{23} \rangle}&& X\times_{G_0} Y \ar[r]& X\otimes_{H} Y.
}
\end{equation}
\item Given 1-cells $X,Y:G\to H$, the 2-cells $f:X\to Y$ are the maps of $G$-$H$-bilocales, with the usual composition.
\item The coherence isomorphisms are, in terms of their inverse image homomorphisms, entirely analogous to those of the bicategory of commutative unital rings.
\end{itemize}
\end{definition}

\begin{lemma}\label{diagonalcoverable}
Let $G$ be an \'etale-covered groupoid, and let $(X,p,\act)$ and $(Y,q,\bct)$ be a left and a right $G$-locale, respectively. Then 
\[
X\otimes_G Y = X\otimes_{\cov G} Y.
\]
\end{lemma}

\begin{proof}
Let $\opens=\opens(G)$ and $\cov \opens=\opens(\cov G)$ be the quantales of $G$ and $\cov G$, respectively, and let $B$ be the base locale. Recall from \cite{Re15}*{Lemma 3.8, Th.\ 3.10} that the coequalizer $X\otimes_{\cov G} Y$ coincides with the frame $X\otimes_{\cov \opens} Y$ and that the latter can be identified with the subframe of elements $\xi\in X\otimes_{B} Y$ such that 
\[
[\cov\pi^*_{12}\circ \act^*,\cov\pi^*_3](\xi)=[\cov\pi^*_1,\cov \pi^*_{23}\circ \bct^*](\xi).
\]
Since $X$ and $Y$ are right and left $\opens$-locales, respectively, we have
\[
\act^*=(\ident_X\otimes j)\circ \phi^*\quad \text{and}\quad \bct^*=(j\otimes \ident_Y)\circ\psi^*,
\]
where $\phi^*:X\to X\otimes_B \opens$ and $\psi^*:Y\to \opens\otimes_B Y$ are frame homomorphisms. Similarly, the coequalizer $X\otimes_G Y$ can be identified with the subframe of elements $\xi\in X\otimes_B Y$ such that
\[
[\pi^*_{12}\circ \phi^*,\pi^*_3](\xi)=[\pi^*_1,\pi^*_{23}\circ \psi^*](\xi).
\]
Then $X\otimes_{\cov G} Y\subset X\otimes_{G} Y$. Indeed, for all $\xi\in X\otimes_{\cov \opens} Y$, we have
\begin{align*}
&(\ident_X\otimes j\otimes \ident_Y)\circ [\pi^*_1,\pi^*_{23}\circ \psi^*](\xi)\\
&=[(\ident_X\otimes j\otimes \ident_Y)\circ \pi^*_1,(\ident_X\otimes j\otimes \ident_Y)\circ\pi^*_{23}\circ \psi^*](\xi)\\
&=[\cov\pi^*_1,\cov \pi^*_{23}\circ \bct^*](\xi)\\
&=[\cov\pi^*_{12}\circ \act^*,\cov \pi^*_3](\xi)& \text{(because $\xi\in X\otimes_{\cov \opens} Y$)} \\
&= [\cov\pi^*_{12}\circ (\ident_X\otimes j)\circ \phi^*,\cov \pi^*_3](\xi)\\
&= [(\ident_X\otimes j\otimes \ident_Y)\circ \pi^*_{12}\circ \phi^*, (\ident_X\otimes j\otimes \ident_Y)\circ \pi^*_3](\xi)\\
&= (\ident_X\otimes j\otimes \ident_Y)\circ [\pi^*_{12}\circ \phi^*, \pi^*_3](\xi).
\end{align*}
Therefore, $[\pi^*_1,\pi^*_{23}\circ \psi^*](\xi)=[\pi^*_{12}\circ \phi^*, \pi^*_3](\xi)$ because $\ident_X\otimes j\otimes \ident_Y$ is monic. Conversely, assume that $\xi\in X\otimes_{G} Y$, that is
\[
[\pi^*_{12}\circ \phi^*,\pi^*_3](\xi)=[\pi^*_1,\pi^*_{23}\circ \psi^*](\xi).
\]
Hence, 
\[
(\ident_X\otimes j\otimes \ident_Y)\circ [\pi^*_1,\pi^*_{23}\circ \psi^*](\xi)=(\ident_X\otimes j\otimes \ident_Y)\circ [\pi^*_{12}\circ \phi^*, \pi^*_3](\xi).
\]
This implies that
\[
[\cov\pi^*_{12}\circ \act^*,\cov\pi^*_3](\xi)=[\cov\pi^*_1,\cov \pi^*_{23}\circ \bct^*](\xi),
\]
which is equivalent to saying that $\xi\in X\otimes_{\cov G} Y$. This proves that $X\otimes_{\cov G} Y=X\otimes_{G} Y$. \qed   
\end{proof}

\begin{definition}\label{QOloc}
Let $\opens_1$ and $\opens_2$ be inverse-embedded quantal frames. By an \emph{$\opens_1$-$\opens_2$-bilocale} $X$ is meant an $\cov\opens_1$-$\cov\opens_2$-bilocale such that $(\alpha_1)_*$ and $(\alpha_2)_*$ factor in $\Frm$ (necessarily uniquely) through $j\otimes\ident_X$ and $\ident_X\otimes j$, respectively, where $(\alpha_i)_*$, for $i=1,2$, is the right adjoint of the module action $\alpha_i$. The \emph{category $\opens_1$-$\opens_2$-$\Loc$ of $\opens_1$-$\opens_2$-bilocales} consists of $\opens_1$-$\opens_2$-bilocales as objects and $\cov\opens_1$-$\cov\opens_2$-bilocale maps as morphisms --- see \cite{Re15}*{Def.\ 4.1}. Therefore, the bicategory $\biQOl$ has the inverse-embedded quantal frames as $0$-cells, bilocales as $1$-cells, and the maps of bilocales as $2$-cells. The composition of $1$-cells $_{{\cov \opens}_1} X_{{\cov \opens}_2}$ and $_{{\cov \opens}_2} Y_{{\cov \opens}_3}$ is defined by
\[
Y\circ X:=X\otimes_{{\cov \opens}_2} Y,
\]
and the coherence morphisms are analogous to those of the bicategory of  unital rings.
\end{definition}

\begin{lemma}
The composition of $1$-cells $_{{\cov \opens}_1} X_{{\cov \opens}_2}$ and $_{{\cov \opens}_2} Y_{{\cov \opens}_3}$, given by
\[
Y\circ X=X\otimes_{{\cov \opens}_2} Y,
\] 
is well defined.
\end{lemma}

\begin{proof}
Let $\opens_1=\opens(G)$, $\opens_2=\opens(H)$, and $\opens_3=\opens(K)$, and also $B_1=\opens(G_0)$, $B_2=\opens(H_0)$, and $B_3=\opens(K_0)$ (the quantales and the base locales of the \'etale-covered groupoids $G$, $H$ and $K$, respectively). Then
\begin{align*}
X\otimes_{{\cov \opens}_2} Y &= X\otimes_{\cov H} Y& \text{(due to \cite{Re15}*{Th.\ 3.10})}\\
&=  X\otimes_{H} Y& \text{(by Lemma~\ref{diagonalcoverable}),}
\end{align*}
and $X\otimes_{{\cov \opens}_2} Y$
is an ${\cov \opens}_1$-${\cov \opens}_3$-bilocale due to \cite{Re15}*{Lemma 4.4}. Let us denote the actions by $\cov\act_1:\cov G_1\times_{G_0} X\to X$ and  $\cov\bct_2:Y\times_{K_0} \cov K_1\to Y$. Then, by assumption, $\cov\act^*_1$ and $\cov\bct^*_2$ factor (uniquely) in $\Frm$ through $j_G\otimes\ident_X$ and $\ident_Y\otimes j_K$. Therefore, the frame homomorphisms
\begin{eqnarray*}
\cov\act^*_1\otimes \ident_Y&:& X\otimes_{\cov\opens_2} Y\longrightarrow \opens_1\otimes_{B_1} (X\otimes_{{\cov \opens}_2} Y)\\
\ident_X\otimes \cov\bct^*_2&:&X\otimes_{\cov\opens_2} Y\longrightarrow (X\otimes_{{\cov \opens}_2} Y)\otimes_{B} \opens_3
\end{eqnarray*}
factor (uniquely) through $j_G\otimes\ident_X\otimes \ident_Y$ and $\ident_X\otimes\ident_Y\otimes j_K$ in $\Frm$.
This proves that $Y\circ X$ defines an $\opens_1$-$\opens_3$-bilocale. \qed 
\end{proof}

\begin{theorem}\label{GHLocisoquantalpairs}
Let $G$ and $H$ be \'etale-covered groupoids. The categories $G$-$H$-$\Loc$ and $\opens(G)$-$\opens(H)$-$\Loc$ are isomorphic.
\end{theorem}
\begin{proof}
Straightforward from Theorem~\ref{funct3, theo: gloccongqoloc} and \cite{Re15}*{Th.\ 4.8}. \qed
\end{proof}

\begin{corollary}\label{biequivalencecoverable}
The bicategories $\biGrpd$ and $\biQOl$ are bi-equivalent.
\end{corollary}

We conclude by showing that for a suitable $\opens_1$-$\opens_2$-bilocale the involute of an element $j(q)\in j(\opens_1)$ can be regarded as an adjoint operator.

\begin{theorem}\label{funct3, lemma:axyxay}
Let $\opens_1$ and $\opens_2$ be inverse-embedded quantal frames, and let $X$ be an $\opens_1$-$\opens_2$-bilocale which is both an open $\opens_1$-locale and an $\opens_2$-sheaf. Then, denoting the sheaf inner product by $\langle-,-\rangle$, we have
\[
\langle j(a)x,y\rangle = \langle x,j(a^*)y\rangle
\]
for all $a\in \opens_1$ and $x,y\in X$.
\end{theorem}

\begin{proof}
This follows directly from \cite{Funct2}*{Th.\ 3.8}:
\[
\langle j(q)s,t\rangle = \V_{\substack{u\in \ipi(\cov \opens_1)\\ u\leq j(q)}}\langle us,t\rangle
= \V_{\substack{u\in \ipi(\cov \opens_1)\\ u\leq j(q)}}\langle s,u^*t\rangle
= \langle s,j(q^*)t\rangle. \qed
\]
\end{proof}

\vspace*{5mm}
\noindent {\sc
Centro de An\'alise Matem\'atica, Geometria e Sistemas Din\^amicos
Departamento de Matem\'{a}tica, Instituto Superior T\'{e}cnico\\
Universidade de Lisboa\\
Av.\ Rovisco Pais 1, 1049-001 Lisboa, Portugal}\\
{\it E-mail:} {\sf quijano.juanpablo@gmail.com}, {\sf pmr@math.tecnico.ulisboa.pt}


\bib{Bun08}{article}{
  author={Buneci, M{\u {a}}d{\u {a}}lina Roxana},
  title={Groupoid categories},
  conference={ title={Perspectives in operator algebras and mathematical physics}, },
  book={ series={Theta Ser. Adv. Math.}, volume={8}, publisher={Theta, Bucharest}, },
  date={2008},
  pages={27--40},
  review={\MR {2433025 (2010b:22007)}},
}

\bib{BEM12}{article}{
  author={Buss, Alcides},
  author={Exel, Ruy},
  author={Meyer, Ralf},
  title={Inverse semigroup actions as groupoid actions},
  journal={Semigroup Forum},
  volume={85},
  date={2012},
  number={2},
  pages={227--243},
  issn={0037-1912},
  review={\MR {2969047}},
  doi={10.1007/s00233-012-9418-y},
}

\bib{EGHK18}{book}{
  author={Eklund, Patrik},
  author={Guti\'{e}rrez Garc\'{\i }a, Javier},
  author={H\"{o}hle, Ulrich},
  author={Kortelainen, Jari},
  title={Semigroups in complete lattices},
  series={Developments in Mathematics},
  volume={54},
  note={Quantales, modules and related topics; With a foreword by Jimmie Lawson},
  publisher={Springer, Cham},
  date={2018},
  pages={xxi+326},
  isbn={978-3-319-78947-7},
  isbn={978-3-319-78948-4},
  review={\MR {3793185}},
  doi={10.1007/978-3-319-78948-4},
}

\bib{HeymansGrQu}{article}{
  author={Heymans, Hans},
  title={Sheaves on involutive quantales: Grothendieck quantales},
  journal={Fuzzy Sets and Systems},
  volume={256},
  date={2014},
  pages={117--148},
  issn={0165-0114},
  review={\MR {3263495}},
  doi={10.1016/j.fss.2013.07.008},
}

\bib{stonespaces}{book}{
  author={Johnstone, Peter T.},
  title={Stone Spaces},
  series={Cambridge Studies in Advanced Mathematics},
  volume={3},
  note={Reprint of the 1982 edition},
  publisher={Cambridge University Press},
  place={Cambridge},
  date={1986},
  pages={xxii+370},
  isbn={0-521-33779-8},
  review={\MR {861951 (87m:54001)}},
}

\bib{JT}{article}{
  author={Joyal, Andr{\'e}},
  author={Tierney, Myles},
  title={An extension of the Galois theory of Grothendieck},
  journal={Mem. Amer. Math. Soc.},
  volume={51},
  date={1984},
  number={309},
  pages={vii+71},
  issn={0065-9266},
  review={\MR {756176 (86d:18002)}},
}

\bib{KL}{article}{
  author={Kudryavtseva, Ganna},
  author={Lawson, Mark V.},
  title={A perspective on non-commutative frame theory},
  journal={Adv. Math.},
  volume={311},
  date={2017},
  pages={378--468},
  issn={0001-8708},
  review={\MR {3628219}},
  doi={10.1016/j.aim.2017.02.028},
}

\bib{Kumjian84}{article}{
  author={Kumjian, Alexander},
  title={On localizations and simple $C^{\ast } $-algebras},
  journal={Pacific J. Math.},
  volume={112},
  date={1984},
  number={1},
  pages={141--192},
  issn={0030-8730},
  review={\MR {739145}},
}

\bib{LL}{article}{
  author={Lawson, Mark V.},
  author={Lenz, Daniel H.},
  title={Pseudogroups and their \'etale groupoids},
  journal={Adv. Math.},
  volume={244},
  date={2013},
  pages={117--170},
  issn={0001-8708},
  review={\MR {3077869}},
  doi={10.1016/j.aim.2013.04.022},
}

\bib{MaRe10}{article}{
  author={Matsnev, Dmitry},
  author={Resende, Pedro},
  title={\'Etale groupoids as germ groupoids and their base extensions},
  journal={Proc. Edinb. Math. Soc. (2)},
  volume={53},
  date={2010},
  number={3},
  pages={765--785},
  issn={0013-0915},
  review={\MR {2720249}},
  doi={10.1017/S001309150800076X},
}

\bib{Paterson}{book}{
  author={Paterson, Alan L. T.},
  title={Groupoids, inverse semigroups, and their operator algebras},
  series={Progress in Mathematics},
  volume={170},
  publisher={Birkh\"auser Boston Inc.},
  place={Boston, MA},
  date={1999},
  pages={xvi+274},
  isbn={0-8176-4051-7},
  review={\MR {1724106 (2001a:22003)}},
}

\bib{PR12}{article}{
  author={Protin, M. Clarence},
  author={Resende, Pedro},
  title={Quantales of open groupoids},
  journal={J. Noncommut. Geom.},
  volume={6},
  date={2012},
  number={2},
  pages={199--247},
  issn={1661-6952},
  review={\MR {2914865}},
  doi={10.4171/JNCG/90},
}

\bib{Funct2}{report}{
  author={Quijano, Juan Pablo},
  author={Resende, Pedro},
  title={Functoriality of groupoid quantales. II},
  date={2018},
  eprint={https://arxiv.org/abs/1803.01075},
}

\bib{QRnonuspp}{article}{
  author={Quijano, Juan Pablo},
  author={Resende, Pedro},
  title={Effective equivalence relations and principal quantales},
  journal={Semigroup Forum},
  volume={99},
  date={2019},
  number={3},
  pages={754--787},
  issn={0037-1912},
  review={\MR {4036870}},
  doi={10.1007/s00233-019-09994-z},
}

\bib{RenaultLNMath}{book}{
  author={Renault, Jean},
  title={A groupoid approach to $C^{\ast } $-algebras},
  series={Lecture Notes in Mathematics},
  volume={793},
  publisher={Springer},
  place={Berlin},
  date={1980},
  pages={ii+160},
  isbn={3-540-09977-8},
  review={\MR {584266 (82h:46075)}},
}

\bib{gamap2006}{article}{
  author={Resende, Pedro},
  title={Lectures on \'{e}tale groupoids, inverse semigroups and quantales},
  conference={ title={SOCRATES Intensive Program 103466-IC-1-2003-1-BE-ERASMUS-IPUC-3 --- GAMAP: Geometric and Algebraic Methods of Physics and Applications}, date={2006}, place={Univ. Antwerp}, },
  eprint={https://www.researchgate.net/publication/265630468},
}

\bib{Re07}{article}{
  author={Resende, Pedro},
  title={\'Etale groupoids and their quantales},
  journal={Adv. Math.},
  volume={208},
  date={2007},
  number={1},
  pages={147--209},
  issn={0001-8708},
  review={\MR {2304314 (2008c:22002)}},
}

\bib{GSQS}{article}{
  author={Resende, Pedro},
  title={Groupoid sheaves as quantale sheaves},
  journal={J. Pure Appl. Algebra},
  volume={216},
  date={2012},
  number={1},
  pages={41--70},
  issn={0022-4049},
  review={\MR {2826418}},
  doi={10.1016/j.jpaa.2011.05.002},
}

\bib{Re15}{article}{
  author={Resende, Pedro},
  title={Functoriality of groupoid quantales. I},
  journal={J. Pure Appl. Algebra},
  volume={219},
  date={2015},
  number={8},
  pages={3089--3109},
  issn={0022-4049},
  review={\MR {3320209}},
  doi={10.1016/j.jpaa.2014.10.004},
}

\bib{RR}{article}{
  author={Resende, Pedro},
  author={Rodrigues, Elias},
  title={Sheaves as modules},
  journal={Appl. Categ. Structures},
  volume={18},
  date={2010},
  number={2},
  pages={199--217},
  issn={0927-2852},
  review={\MR {2601963}},
  doi={10.1007/s10485-008-9131-x (on-line 2008)},
}

\bib{Rosenthal1}{book}{
  author={Rosenthal, Kimmo I.},
  title={Quantales and Their Applications},
  series={Pitman Research Notes in Mathematics Series},
  volume={234},
  publisher={Longman Scientific \& Technical},
  place={Harlow},
  date={1990},
  pages={x+165},
  isbn={0-582-06423-6},
  review={\MR {1088258 (92e:06028)}},
}



\begin{bibdiv}

\begin{biblist}

\bibselect{bibliography}

\end{biblist}

\end{bibdiv}
\end{document}